\documentclass[a4paper,11pt]{article}
\usepackage{amssymb,amsmath,amsthm}
\usepackage{graphics,graphicx,color}
\usepackage{srcltx, bm}
\DeclareGraphicsExtensions{.eps} \textwidth=170mm
\definecolor{darkred}{rgb}{0.9,0.1,0.1}

\newtheorem{proposition}{Proposition}[section]
\newtheorem{theorem}{Theorem}[section]
\newtheorem{lemma}{Lemma}[section]
\newtheorem{corollary}{Corollary}[section]

{\rm}
\definecolor{darkred}{rgb}{0.9,0.1,0.1}

\def\eps{\varepsilon}
\def\eps{\varepsilon}
\definecolor{darkred}{rgb}{0.9,0.1,0.1}

\begin{document}

\title{Stochastic homogenization of convolution type operators}

\author{  A.~Piatnitski$^{\circ,\sharp}$,
 E.~Zhizhina$^\sharp$}
\date{}
\maketitle

\parskip 0.04 truein

\begin{center}
$^\sharp$
Institute for Information Transmission Problems RAS\\
Bolshoi Karetny per., 19,
Moscow, 127051, Russia
\end{center}

\begin{center}
$^\circ$
Arctic University of Norway, UiT,
campus  Narvik,\\
Postbox 385, 8505 Narvik, Norway\\
\end{center}

\begin{abstract}
This paper deals with the homogenization problem for convolution type non-local operators
in  random statistically homogeneous ergodic media.  Assuming that the convolution kernel
has a finite second moment and satisfies  the uniform ellipticity and certain symmetry
conditions, we prove the almost sure homogenization result and show that the limit operator is
a second order elliptic differential operator with constant deterministic coefficients.
\end{abstract}



\section{Introduction}

The paper deals with homogenization problem for integral operators of convolution type in $\mathbb R^d$ with dispersal kernels
that have random statistically homogeneous ergodic coefficients. For such operators, under natural  integrability, moment and uniform
ellipticity conditions as well as the symmetry condition we  prove the homogenization result and study the properties of the limit operator.

The integral operators with a kernel of convolution type are of great interest both from the mathematical point of view and due to various important applications in other fields.  Among such applications are  models of population dynamics and ecological models, see \cite{OFetal}, \cite{DEE} and references therein,
 non-local diffusion problems, see  \cite{AMRT, BCF},
   continuous particle systems, see  \cite{ FKK, KPZ},  image processing algorithms, see \cite{GiOs}.  In the cited works only the case of homogeneous environments has been considered.  In this case the corresponding dispersal kernel depends only on the displacement $y-x$. However, many applications deal with non-homogeneous environments.
Such environments are described in terms of integral operator  whose dispersal kernels  depend not only on the displacement $x-y$ but also
on the starting and the ending positions $x, y$.

When studying the large-time behaviour of evolution processes in these environments it is natural to make the diffusive scaling
in the corresponding integral operators and to consider the homogenization problem for the obtained family of operators with a small
positive parameter. In what follows we call this parameter $\eps$

The case of environments with periodic characteristics has been studied in the recent work
\cite{PiZhi17}. It has been shown that under natural moment and symmetry conditions on the kernel the family of rescaled operators admits homogenization, and that for the corresponding jump Markov process the Central Limit Theorem and the Invariance Principle hold.   Interesting homogenization problems for periodic operators containing both second order elliptic operator
and nonlocal Levy type operator have been considered in \cite{Arisawa} and \cite{Sandric2016}.

In the present paper we consider the more realistic case of environments with random statistically homogeneous characteristics.
More precisely, we assume that the dispersal kernel of the studied operators has the form $\Lambda(x,y)a(x-y)$, $x,\,y\in\mathbb R^d$,
where $a(z)$ is a deterministic even function that belongs to $L^1(\mathbb R^d)\cap L^2_{\rm loc}(\mathbb R^d)$ and has finite second moments,
while $\Lambda(x,y)=\Lambda(x,y,\omega)$ is a statistically homogeneous symmetric ergodic random field that satisfies the uniform ellipticity conditions
$0<\Lambda^-\leq \Lambda(x,y)\leq \Lambda^+$.\\
 Making  a diffusive scaling  we obtain the family of operators
\begin{equation}\label{L_u_biseps}
(L^\eps u)(x) \ = \ \eps^{-d-2} \int\limits_{\mathbb R^d} a\Big(\frac{x-y}{\eps}\Big)
\Lambda\Big(\frac{x}{\eps},\frac{y}{\eps}\Big) (u(y) - u(x)) dy,
\end{equation}
 where a positive scaling factor $\eps$ is a parameter.

For the presentation simplicity we assume in this paper that $\Lambda(x,y)=\mu(x)\mu(y)$ with a statistically homogeneous ergodic
field $\mu$. However, all our results remain valid for the generic statistically homogeneous symmetric random fields $\Lambda(x,y)$
that satisfy the above ellipticity conditions.




The main  goal of this work is to investigate  
the limit behaviour of $L^\eps$ as $\eps\to 0$.
We are going to show that the family $L^\eps$ converges almost surely to a second order
elliptic operator with constant deterministic coefficient in the so-called $G$-topology, that is
for any $m>0$ the family of operators $(-L^\eps+m)^{-1}$ almost surely converges strongly in $L^2(\mathbb R^d)$
to the operator $(-L^0+m)^{-1}$ where $L^0=\Theta^{ij}\frac{\partial^2}{\partial x^i\partial x^j}$,
and $\Theta$ is a positive definite constant matrix.

There is a vast existing literature devoted to
homogenization theory of differential operators, at present it is a well-developed area, see for instance  monographs \cite{BLP} and \cite{JKO}.   The first homogenization results for divergence form differential operators with random coefficients were obtained
in pioneer works \cite{Ko78} and \cite{PaVa79}.  In these works it was shown that the generic divergence form second order elliptic
operator with random statistically homogeneous coefficients admits homogenization. Moreover, the limit operator has constant coefficients,
in the ergodic case these coefficients are deterministic.

Later on a number of important homogenization results have been obtained for various elliptic and parabolic differential equations and system of equations   in random stationary media. The reader can find many references in the book  \cite{JKO}.




Homogenization of elliptic difference schemes and discrete operators in statistically homogeneous media  has been performed in
\cite{Ko87}, \cite{Ko86}. Also, in \cite{Ko86} several limit theorems have been proved for random walks in stationary discrete
random media that possess different types of symmetry.

To our best knowledge in the existing literature there are no results on stochastic  homogenization of convolution type integral operators with a dispersal kernel  that has stationary rapidly oscillating coefficients.

In the one-dimensional case a homogenization problem for  the operators that have both local and non-local parts has been
considered in the work \cite{Rho_Var2008}. This work  deals with scaling limits of the solutions to stochastic differential
equations in dimension one with stationary coefficients driven by Poisson random measures and Brownian
motions. The annealed convergence theorem is proved, in which the limit exhibits
a diffusive or superdiffusive behavior, depending on whether  the Poisson random measure has a finite second moment
or not. It is important in this paper that the diffusion coefficient does not degenerate.


Our approach  relies on asymptotic expansion techniques and using the so-called corrector. As  often happens in the case of
random environments we cannot claim the existence of a stationary corrector. Instead, we construct a corrector which is a random field
in $\mathbb R^d$ with stationary increments and   almost surely has a sublinear growth in $L^2(\mathbb R^d)$.  \\
When substituting two leading terms of the expansion for the solution of the original equation, we obtain 
the discrepancies being oscillating functions with zero average.
Some of these functions are not stationary.
 In order to show that the contributions of these discrepancies
are asymptotically negligible we add to the expansion two extra terms. The necessity of constructing these terms is essentially
related to the fact that, in contrast with the case of elliptic differential equations, the resolvent of the studied operator is not
locally compact in $L^2(\mathbb R^d)$.

The paper is organized as follows:

 In Section \ref{s_pbmset} we provide the detailed setting of the problem
and formulate the main result of this work.

The leading terms of the ansatz for a solution of  equation $(L^\eps-m)u^\eps=f$ with  $f\in C_0^\infty(\mathbb R^d)$ are introduced in Section \ref{s_asyexp}.  Also in this section we outline the main steps of the proof of  our homogenization theorem.

Then in Section \ref{s_corr} we construct the principal corrector in the asymptotic expansion and study the properties
of this corrector.

Section \ref{s_addterms} is devoted to constructing two additional terms of the expansion of $u^\eps$. Then we introduce
the effective matrix and prove its positive definiteness.

Estimates for the remainder in the asymptotic expansion are obtained in Section \ref{s_estrem}.

Finally, in Section \ref{s_proofmain} we complete the proof of the homogenization theorem.

\section{Problem setup and  main result}\label{s_pbmset}

\noindent
We consider a homogenization problem for a random convolution type operator of the form
\begin{equation}\label{L_u}
(L_\omega u)(x) \ = \ \mu(x,\omega) \int\limits_{\mathbb R^d} a(x-y) \mu(y,\omega) (u(y) - u(x)) dy.
\end{equation}
For the function $a(z)$ we assume the following:
\begin{equation}\label{A1}
a(z) \in  L^{1}(\mathbb R^d) \cap L^{2}_{\rm loc}(\mathbb R^d), \quad  a(z) \ge 0; \quad  a(-z) = a(z),
\end{equation}
and
\begin{equation}\label{M2}
\| a \|_{L^1(\mathbb R^d)}  = \int\limits_{\mathbb R^d} a(z) \ dz = a_1 < \infty; \quad \sigma^2 =  \int\limits_{\mathbb R^d} |z|^2 a(z) \ dz < \infty.
\end{equation}
We also assume that
\begin{equation}\label{add}
\mbox{there exists a constant} \; c_0>0 \; \mbox{ and a cube } \; {\bf B} \subset  \mathbb R^d, \; \mbox{ such that } \; a(z) \ge c_0 \quad \mbox{for all } \; z \in {\bf B}.
\end{equation}
This additional condition on $a(z)$ is naturally satisfied for regular kernels, and we introduced \eqref{add} for a presentation simplicity. Assumption \eqref{add} essentially simplifies derivation of inequality \eqref{L2B}, on which the proof of the smallness of the first corrector is based, see Proposition \ref{1corrector} below. We notice that inequality \eqref{L2B} can also  be derived without assumption \eqref{add}, however in this case  additional arguments of measure theory are required.
\\[5pt]
Let $(\Omega,\mathcal{F}, \mathbb P)$ be a standard probability space.
We assume that the random field $ \mu(x,\omega)= {\bm\mu} (T_x \omega) $ is stationary  and bounded from above and from below:
\begin{equation}\label{lm}
0< \alpha_1 \le    \mu(x,\omega) \le \alpha_2 < \infty;
\end{equation}
here  ${\bm\mu} (\omega) $ is a random variable, and $T_x$, $x\in \mathbb R^d$, is an ergodic group of measurable transformations  acting in $\omega$-space $\Omega$, $T_x:\Omega
\mapsto\Omega$, and possessing the following properties:
\begin{itemize}
  \item  $T_{x+y}=T_x\circ T_y\quad\hbox{for all }x,\,y\in\mathbb R^d,\quad T_0={\rm Id}$,
  \item   $\mathbb P(A)=\mathbb P(T_xA)$ for any $A\in\mathcal{F}$ and any $x\in\mathbb R^d$,
  \item $T_x$ is a measurable map from $\mathbb R^d\times \Omega$ to $\Omega$, where $\mathbb R^d$ is equipped
  with the Borel $\sigma$-algebra.
\end{itemize}

\medskip

Let us consider a family of the following operators
\begin{equation}\label{L_eps}
(L^{\varepsilon}_\omega u)(x) \ = \ \frac{1}{\varepsilon^{d+2}} \int\limits_{\mathbb R^d} a \Big( \frac{x-y}{\varepsilon} \Big) \mu \Big( \frac{x}{\varepsilon},\omega \Big) \mu \Big( \frac{y}{\varepsilon},\omega \Big) \Big( u(y) - u(x) \Big) dy.
\end{equation}
We are interested in the limit behavior of the operators $L^{\varepsilon}_\omega$ as $\varepsilon \to 0$ .
We are going to show that for a.e. $\omega$ the operators  $L^{\varepsilon}_\omega$ converge to a differential operator with constant coefficients in the topology of the resolvent convergence. Let us fix $m>0$, any $f \in L^2(\mathbb R^d)$, and define $u^{\varepsilon}$ as the solution of equation:
\begin{equation}\label{u_eps}
(L^{\varepsilon}_\omega - m) u^{\varepsilon} \ = \ f, \quad \mbox{ i.e. } \; u^{\varepsilon} \ = \ (L^{\varepsilon}_\omega - m)^{-1} f
\end{equation}
with $f \in L^2(\mathbb R^d)$. Denote by $\hat L$ the following operator in  $L^2(\mathbb R^d)$:
\begin{equation}\label{L_hat}
\hat L u \ = \ \sum_{i,j = 1}^d \Theta_{i j} \frac{\partial^2 u}{\partial x_i  \ \partial x_j}, \quad {\cal D}(\hat L) = H^2(\mathbb R^d)
\end{equation}
with a positive definite matrix $\Theta = \{ \Theta_{i j} \}, \ i,j = 1, \ldots, d,$ defined below, see (\ref{Positive}). Let $u_0(x)$ be the solution of equation
\begin{equation}\label{u_0}
\sum_{i,j = 1}^d \Theta_{i j} \frac{\partial^2 u_0}{\partial x_i  \ \partial x_j} - m u_0 = f,  \quad \mbox{ i.e. } \; u_0 \ = \ (\hat L - m)^{-1} f
\end{equation}
with the same right-hand side $f$ as in (\ref{u_eps}).

\begin{theorem}\label{T1} Almost surely for any $f \in L^2(\mathbb R^d)$  and any $m>0$ the convergence holds:
\begin{equation}\label{t1}
\| (L^{\varepsilon}_\omega - m)^{-1} f - (\hat L - m)^{-1} f \|_{L^2(\mathbb R^d)} \ \to 0 \quad \mbox{  as } \; \varepsilon \to 0.
\end{equation}

\end{theorem}
The statement of Theorem \ref{T1} remains valid in the case of non-symmetric operators $L^\eps$ of the form
\begin{equation}\label{L_eps_ns}
(L^{\varepsilon,{\rm ns}}_\omega u)(x) \ = \ \frac{1}{\varepsilon^{d+2}} \int\limits_{\mathbb R^d} a \Big( \frac{x-y}{\varepsilon} \Big) \lambda \Big( \frac{x}{\varepsilon},\omega \Big) \mu \Big( \frac{y}{\varepsilon},\omega \Big) \Big( u(y) - u(x) \Big) dy
\end{equation}
with $\lambda(z,\omega)=\bm{\lambda}(T_z\omega)$ such that $0< \alpha_1 \le    \lambda(x,\omega) \le \alpha_2 < \infty$.
In this case the equation \eqref{u_eps} reads
\begin{equation}\label{u_eps_nssss}
(L^{\varepsilon,{\rm ns}}_\omega - m) u^{\varepsilon} \ = \ f.
\end{equation}

\begin{corollary}\label{cor_main}
 Let $\lambda(z,\omega)$  and $\mu(z,\omega)$ satisfy condition \eqref{lm}. Then  a.s. for any $f\in L^2(\mathbb R^d)$ and any $m>0$ the limit relation in \eqref{t1} holds true with $\hat L^{\rm ns} u \ = \ \sum_{i,j = 1}^d \Theta^{\rm ns}_{i j} \frac{\partial^2 u}{\partial x_i  \ \partial x_j}$, \
 $\Theta^{\rm ns}=\big(\mathbb E \big\{\frac{\bm\mu}{\bm\lambda}\big\}\big)^{-1} \Theta$, and $\Theta$ defined in
 \eqref{Positive}.
\end{corollary}

\section{Asymptotic expansion for $u^\eps$ }\label{s_asyexp}

We begin this section by introducing a set of functions $f \in C_0^\infty(\mathbb R^d)$ such that
$u_0 \ = \ (\hat L - m)^{-1} f\in C_0^\infty(\mathbb R^d)$.  We denote this set by $ {\cal S}_0(\mathbb R^d)$. Observe that
this set is dense in $L^2(\mathbb R^d)$. Indeed, if we take  $\varphi(x)\in C^\infty(\mathbb R)$ such that $0\leq\varphi \leq 1$,
$\varphi=1$ for $x\leq 0$ and  $\varphi=0$ for $x\geq 1$, then letting  $f_n=(\hat L-m)\big(\varphi(|x|-n)(\hat L-m)^{-1}f(x)\big)$
one can easily check that $f_n\in C_0^\infty(\mathbb R^d)$ and  $\|f_n-f\|_{L^2(\mathbb R^d)}\to0$, as $n\to\infty$.\\
We consider first the case when $f \in {\cal S}_0(\mathbb R^d)$
and denote by $Q$ a cube centered at the origin and such that  $\mathrm{supp}(u_0)\subset Q$.
We want to
prove the convergence
\begin{equation}\label{convergence1}
\| u^{\varepsilon} - u_0 \|_{L^2(\mathbb R^d)} \ \to 0, \quad \mbox{ as } \ \varepsilon \to 0,
\end{equation}
where the functions  $u^\varepsilon$ and $u_0$ are defined in  (\ref{u_eps}) and (\ref{u_0}), respectively.
To this end we approximate the function $ u^\varepsilon (x, \omega)$ by means of the following ansatz
\begin{equation}\label{v_eps}
w^{\varepsilon}(x, \omega) \ = \ v^\varepsilon (x, \omega) + u_2^\varepsilon (x, \omega) + u_3^\varepsilon(x, \omega), \quad \mbox{ with } \;  v^{\varepsilon}(x, \omega) \ = \ u_0(x)+ \varepsilon \theta \big(\frac{x}{\varepsilon}, \omega\big) \nabla u_0(x),
\end{equation}
where $\theta \big(z, \omega\big) $ is a vector function which is often called a corrector. It will be introduced later on as a solution of an auxiliary problem that does not depend on $\eps$, see \eqref{korrkappa1}.  A solution of this problem, $\theta(z,\omega)$ say, is defined up to an additive
constant vector. \\ We set
\begin{equation}\label{hi}
\chi^\varepsilon (z,\omega) = \theta (z,\omega)+ c^\varepsilon (\omega), \quad  c^\varepsilon (\omega) = - \frac{1}{|Q|} \int\limits_Q  \theta \big( \frac{x}{\varepsilon},\omega \big) dx.
\end{equation}
Observe that under such a choice of the vector $c^\eps$
the function $\chi^\varepsilon \big(\frac x\eps,\omega\big)$ has zero average in $Q$. We show in Proposition \ref{1corrector} that $\eps c^\eps\to 0$ a.s.
  It should be emphasized that $\theta (y, \omega)$ need not be a stationary field, that is we do not claim that  $\theta(y, \omega) = {\bm\theta} (T_y \omega)$ for some random vector ${\bm\theta}(\omega)$.

Two other functions,   $u_2^\varepsilon$ and $u_3^\varepsilon$, that appear in the ansatz in \eqref{v_eps}  will be introduced in \eqref{corr-u2}, \eqref{u3}, respectively.

After substitution $v_\eps$ for $u$ to (\ref{L_eps}) we get
$$
(L^{\varepsilon} v^{\varepsilon})(x) \ = \ \frac{1}{\varepsilon^{d+2}} \int\limits_{\mathbb R^d} a \big( \frac{x-y}{\varepsilon} \big) \mu \big( \frac{x}{\varepsilon} \big) \mu \big( \frac{y}{\varepsilon} \big)
\Big( u_0(y)+ \varepsilon \theta \big(\frac{y}{\varepsilon}\big) \nabla u_0(y)
- u_0(x)-\varepsilon \theta \big(\frac{x}{\varepsilon} \big) \nabla u_0(x)
\Big) dy;
$$
here and in what follows we drop the argument $\omega$ in the random fields $\mu(y,\omega)$, $\theta(y,\omega)$, etc.,
if it does not lead to ambiguity.
After change of variables $\frac{x-y}{\varepsilon}=z$ we get
\begin{equation}\label{ml_1}
(L^{\varepsilon} v^{\varepsilon})(x) \ = \ \frac{1}{\varepsilon^{2}} \int\limits_{\mathbb R^d} dz \  a (z) \mu \big( \frac{x}{\varepsilon} \big) \mu \big( \frac{x}{\varepsilon} -z \big) \Big( u_0(x-\varepsilon z) - u_0(x) + \varepsilon \theta \big(\frac{x}{\varepsilon}-z \big) \nabla u_0 (x-\varepsilon z)  -\varepsilon \theta \big( \frac{x}{\varepsilon} \big) \nabla u_0(x) \Big).
\end{equation}
 The Taylor expansion of a function $u(y)$ with a remainder in  the integral form reads
$$
\begin{array}{c}
u(y) \ = \ u(x) + \int_0^1 \nabla u (x + (y-x)t) \cdot (y-x) \ dt \\[3pt]
= \ u(x) + \nabla u(x) \cdot (y-x) + \int_0^1  \nabla \nabla u(x+(y-x)t) (y-x) (y-x) (1-t) \ dt
\end{array}
$$
and is valid for any $x, y \in \mathbb R^d$. Thus we  can rewrite (\ref{ml_1}) as follows
\begin{eqnarray}
(L^{\varepsilon} v^{\varepsilon})(x) \hskip -1.7cm &&\nonumber\\[1.6mm]
\label{K2_1}
&&\!\!\!\!\!=\, \frac{1}{\varepsilon} \mu \Big( \frac{x}{\varepsilon}, \omega \Big)\nabla u_0(x)\! \cdot\! \int\limits_{\mathbb R^d}  \Big[ -z + \theta \Big(\frac{x}{\varepsilon}-z, \omega \Big) - \theta \Big(\frac{x}{\varepsilon}, \omega \Big) \Big]  a (z) \mu \Big( \frac{x}{\varepsilon} -z, \omega \Big) \, dz
\\[1mm]
\nonumber
&&\!\!\!\!\! +\,\mu \Big(\! \frac{x}{\varepsilon}, \omega \Big) \nabla \nabla u_0 (x)\!\cdot\!  \int\limits_{\mathbb R^d}\! \Big[ \frac12 z\!\otimes\!z\! - z \!\otimes\!\theta \Big(\frac{x}{\varepsilon}\!-\!z,\omega \Big)  \Big]   a (z) \mu \Big( \frac{x}{\varepsilon}\! -\!z, \omega \Big) \, dz
 +\, \ \phi_\varepsilon (x) \hfill\\
 \nonumber
 &&=: \frac{1}{\varepsilon} I^\varepsilon_{-1} + \varepsilon^0 I^\varepsilon_0 + \phi_\varepsilon
\end{eqnarray}
with
\begin{equation}\label{14}
\begin{array}{rl} \displaystyle
\!\!\!\!&\hbox{ }\!\!\!\!\!\!\!\!\!\!\!\!\phi_\varepsilon (x, \omega) =\\[3mm]
& \!\!\!\!\!\!\!\!\displaystyle
\!\! \int\limits_{\mathbb R^d}\!  a (z) \mu \Big( \frac{x}{\varepsilon},\omega \Big) \mu \Big( \frac{x}{\varepsilon}\! -\!z,\omega \Big)  \bigg(\int\limits_0^{1}  \nabla \nabla u_0(x-\varepsilon z t) \!\cdot\! z\!\otimes\!z \,(1-t) \ dt  - \frac{1}{2} \nabla \nabla u_0(x)\!\cdot\!  z\!\otimes\!z \bigg) \, dz
\\[4mm]  &\!\!\!\!\!\!\!\!\! \displaystyle
+\, \frac{1}{\varepsilon} \mu \Big( \frac{x}{\varepsilon},\omega \Big)  \int\limits_{\mathbb R^d} \ a (z)  \mu \Big( \frac{x}{\varepsilon} -z, \omega \Big)  \theta \Big(\frac{x}{\varepsilon}\!-\!z,\omega \Big)\! \Big(\nabla u_0(x- \varepsilon z) - \nabla u_0(x) \Big)\, dz
\\[4mm]  &\!\!\!\!\!\!\!\!\! \displaystyle
+ \mu \Big( \frac{x}{\varepsilon},\omega \Big) \nabla \nabla u_0(x)  \int\limits_{\mathbb R^d} \ a (z)  \mu \Big( \frac{x}{\varepsilon} -z, \omega \Big) z \otimes \theta \Big(\frac{x}{\varepsilon}\!-\!z,\omega \Big)\, dz.
\end{array}
\end{equation}
Here and in what follows $z\otimes z$ stands for the matrix $\{z_iz_j\}_{i,j=1}^d$.

Let us outline the main steps of the proof of relation \eqref{convergence1}.
In order to make the term $I^\eps_{-1}$ in \eqref{K2_1} equal to zero, we should
construct a random  field $\theta \big(z, \omega\big)$ that satisfies the following equation
\begin{equation}\label{korr1}
\int\limits_{\mathbb R^d}  \Big( -z + \theta \big(\frac{x}{\varepsilon}-z, \omega \big) - \theta \big(\frac{x}{\varepsilon}, \omega\big) \Big) \, a (z) \mu \big( \frac{x}{\varepsilon} -z,\omega \big) \ dz \ = \ 0.
\end{equation}
The goal of the first step is to construct such a random field $\theta(z,\omega)$.
Next we show that the second term $I^\varepsilon_0$ can be represented as a sum
$$
I^\varepsilon_0 = \hat L u_0 + S\Big(\frac x\eps,\omega\Big)\nabla\nabla u_0 + f_2^\varepsilon (x,\omega),
$$
where $S(z,\omega)$ is a stationary matrix-field with zero average, and $f_2^\varepsilon (x,\omega)$ is a non-stationary term; both of them are introduced below. We define $u_2^\varepsilon$ and $u_3^\varepsilon$ by
$$
(L^\varepsilon - m) u_2^\varepsilon = - S\Big(\frac x\eps,\omega\Big)\nabla\nabla u_0, \quad  (L^\varepsilon - m) u_3^\varepsilon = - f_2^\varepsilon (x,\omega),
$$
and prove that $\| u_2^\varepsilon \|_{L^2(\mathbb R^d)} \to 0$,  $\| u_3^\varepsilon \|_{L^2(\mathbb R^d)} \to 0$. Then
considering the properties of the corrector $\theta$, see Theorem \ref{t_corrector}, we derive the limit relation
 $\|\varepsilon \theta\big(\frac x\eps\big) \nabla u_0(x) \|_{L^2(\mathbb R^d)} \to 0$,  as $\varepsilon \to 0$.
This yields  $\| w^\varepsilon - u_0 \| \to 0$.

With this choice of  $\theta$, $u_2^\varepsilon$ and  $u_3^\varepsilon$ the expression $(L^\varepsilon - m) w^\varepsilon$ can be rearranged as follows:
$$
(L^\varepsilon - m) w^\varepsilon = (L^\varepsilon - m) v^\varepsilon + (L^\varepsilon - m) (u_2^\varepsilon + u_3^\varepsilon) =
(\hat L - m) u_0 + \phi_\varepsilon - m \varepsilon \theta \nabla u_0
$$
$$
= f  + \phi_\varepsilon - m \varepsilon \theta \nabla u_0 = (L^\varepsilon - m) u^\varepsilon  + \phi_\varepsilon - m \varepsilon \theta \nabla u_0.
$$
We prove below in Lemma \ref{reminder} that  $\|\phi_\varepsilon\|\big._{L^2(\mathbb R^d)}$ is vanishing as $\varepsilon \to 0$.
This implies the convergence $\| w^\varepsilon - u^\varepsilon \|\big._{L^2(\mathbb R^d)} \to 0$ and, by the triangle inequality, the required relation in \eqref{convergence1}.


\medskip



\section{First corrector}\label{s_corr}

\medskip

 In this Section we construct a solution of equation \eqref{korr1}.
Denote
\begin{equation}\label{fkorr1}
r \big(\frac{x}{\varepsilon}, \omega\big) = \int\limits_{\mathbb R^d}  z \, a (z) \, \mu \big( \frac{x}{\varepsilon} -z,\omega \big) \ dz,
\end{equation}
 then $r(\xi, \omega) = \mathbf{r}(T_\xi \omega), \; \xi =  \frac{x}{\varepsilon},$ is a stationary field. Moreover, since $\mathbb{E} \mu ( \xi -z,\omega )= \mathbb{E}{\bm\mu}(T_{\xi-z} \omega) =  const$ for all $z$, then
$$
\mathbb{E} r (\xi, \omega) = \int\limits_{\mathbb R^d}  z \, a (z) \, \mathbb{E}\mu ( \xi -z,\omega ) \ dz \ = \ 0.
$$
Equation  \eqref{korr1} takes the form 
\begin{equation}\label{korrkappa1}
r (\xi, \omega) \ = \ \int\limits_{\mathbb R^d}  a (z) \mu ( \xi -z,\omega ) \, \big( \theta (\xi-z, \omega ) - \theta (\xi, \omega) \big)  \ dz.
\end{equation}
We are going to show now that equation \eqref{korrkappa1} has a solution that possesses the following properties: \\[1.5mm]
{\bf A}) the increments $\zeta_z(\xi, \omega)
= \theta (z+\xi, \omega ) - \theta (\xi, \omega)$ are stationary for any given $z$, i.e.
$$\zeta_z(\xi,  \omega) = \zeta_z(0, T_\xi \omega);$$
{\bf B})
$\eps \theta\big(\frac x\eps,\omega\big) $
is a function of sub-linear growth in $L_{\rm loc}^2(\mathbb R^d)$: for any bounded Lipschitz domain $Q\subset \mathbb R^d$
$$
\Big\|  \varepsilon \, \theta \big(\frac{x}{\varepsilon}, \omega \big) \Big\|_{L^2(Q)} \to 0 \quad \mbox{a.s.} \; \omega \in \Omega.
$$
Here and in the sequel for presentation simplicity we write for the $L^2$ norm of  a vector-function just $L^2(Q)$ instead of
$L^2(Q\,;\,\mathbb R^d)$.


%
%
\medskip

\begin{theorem}\label{t_corrector}
There exists a unique (up to an additive constant vector) solution $\theta\in L^2_{\rm loc}(\mathbb R^d)$ of equation \eqref{korrkappa1} that satisfies conditions {\bf A}{\rm )} -- {\bf B}{\rm )}.
\end{theorem}


\begin{proof}[Proof of Theorem \ref{t_corrector}]
We divide the proof into several steps.\\
{\sl Step 1.} Consider the following operator
 acting in $L^2(\Omega)$:
\begin{equation}\label{A-omega}
(A \varphi)(\omega) = \int\limits_{\mathbb R^d} a(z) {\bm\mu}(T_z \omega) \big( \varphi (T_z \omega) - \varphi(\omega) \big) dz
\end{equation}

\begin{proposition}\label{spectrA}
The spectrum $\sigma(A) \subset (-\infty, 0]$.
\end{proposition}
\begin{proof}
It is straightforward to check that the operator $A$ is bounded and symmetric in the weighted space
$L^2(\Omega, P_\mu) = L^2_\mu(\Omega)$ with $d P_\mu(\omega) = {\bm\mu}(\omega) d P(\omega)$.
Denoting $\tilde \omega = T_z \omega, \ s=-z$, using stationarity of $\mu$ and considering the relation $a(-z) = a(z)$ we get
\begin{equation}\label{PropA1}
\begin{array}{c}
\displaystyle
\int\limits_\Omega \int\limits_{\mathbb R^d} a(z){\bm\mu}(T_z \omega){\bm\mu}(\omega) \varphi^2(T_z \omega) \, dz \, dP(\omega)=
\int\limits_\Omega \int\limits_{\mathbb R^d} a(z) {\bm\mu}(\tilde \omega) {\bm\mu}(T_{-z} \tilde\omega)  \varphi^2(\tilde\omega) \, dz \, dP(\tilde\omega) \\[3pt]
\displaystyle
= \int\limits_\Omega \int\limits_{\mathbb R^d} a(s){\bm\mu}( \omega){\bm\mu}(T_s \omega) \varphi^2(\omega)\,  ds \, dP(\omega).
\end{array}
\end{equation}
Thus
\begin{equation}\label{PropA1bis}
\begin{array}{c}
\displaystyle
\big( A\varphi, \varphi \big)_{L^2_\mu} = \int\limits_\Omega \int\limits_{\mathbb R^d} a(z) {\bm\mu}(T_z \omega) \big( \varphi(T_z \omega) - \varphi(\omega) \big) \varphi(\omega) {\bm\mu}(\omega)  dz dP(\omega)
\\ \displaystyle
= -\frac12 \int\limits_\Omega \int\limits_{\mathbb R^d} a(z) {\bm\mu}(T_z \omega) {\bm\mu}(\omega) \big( \varphi(T_z \omega) - \varphi(\omega) \big)^2  dz dP(\omega)<0.
\end{array}
\end{equation}
Since the norms in $L^2(\Omega)$ and $L^2_\mu(\Omega)$ are equivalent, the desired statement follows.
\end{proof}

Let us consider for any $\delta>0$ the equation
\begin{equation}\label{A-delta}
\delta \varphi(\omega) - \int\limits_{\mathbb R^d} a (z) {\bm\mu} ( T_z \omega ) ( \varphi (T_z \omega ) - \varphi ( \omega) )  \ dz =  r(\omega), \quad r(\omega) =  \int\limits_{\mathbb R^d}  z  a (z) {\bm\mu} (T_z \omega ) \ dz.
\end{equation}
By Proposition \ref{spectrA} the operator $(\delta I - A)^{-1}$ is bounded, then there exists a unique solution $\varkappa^\delta (\omega) = -(\delta I - A)^{-1} r (\omega)$ of \eqref{A-delta}. 
 For any given $z \in R^d$
we set
$$
 u^\delta(z,\omega) =  \varkappa^\delta(T_z \omega) - \varkappa^\delta(\omega).
$$
Then
\begin{equation}\label{u-delta}
u^\delta(z_1 + z_2,\omega) = u^\delta(z_2,\omega) + u^\delta(z_1, T_{z_2} \omega) \quad \forall \ z_1, z_2 \in \mathbb R^d.
\end{equation}
For any  $\xi \in\mathbb R^d$ as an immediate consequence of \eqref{A-delta} we have
\begin{equation}\label{A-delta-xi}
\delta \varkappa^\delta (T_\xi \omega) - \int\limits_{\mathbb R^d} a (z) {\bm\mu} ( T_{\xi+z} \omega ) ( \varkappa^\delta (T_{\xi+z} \omega ) - \varkappa^\delta ( T_\xi \omega) )  \ dz =   \int\limits_{\mathbb R^d}  z  a (z) {\bm\mu} (T_{\xi+z} \omega ) \ dz.
\end{equation}


\medskip

Next we obtain a priori estimates for $\| \varkappa^\delta (T_z \omega) - \varkappa^\delta (\omega)\|_{L^2_M}$ with $dM(z, \omega) = a(z) dz dP(\omega)$.
\begin{proposition}\label{boundM}
The following  estimate holds:
\begin{equation}\label{AB}
\| u^\delta(z,\omega) \|_{L^2_M} = \| \varkappa^\delta (T_z \omega) - \varkappa^\delta (\omega) \|_{L^2_M} \ \le \ C
\end{equation}
with a constant $C$ that does not depend on $\delta$.
\end{proposition}
\begin{proof}
Multiplying equation \eqref{A-delta} by $\varphi(\omega)={\bm\mu}(\omega)\varkappa^\delta(\omega)$ and integrating
the resulting relation over $\Omega$ yields
\begin{equation}\label{Prop2}
\begin{array}{c}
\displaystyle
\delta \int\limits_\Omega \big(\varkappa^\delta(\omega)\big)^2{\bm\mu}(\omega)\, dP(\omega)
 - \int\limits_{\mathbb R^d} \int\limits_\Omega a (z) {\bm\mu} ( T_z \omega ) \big( \varkappa^\delta (T_z \omega ) - \varkappa^\delta ( \omega) \big) \varkappa^\delta(\omega){\bm\mu}(\omega) \,  dz \, dP(\omega) \\ \displaystyle
 =   \int\limits_{\mathbb R^d} \int\limits_\Omega z a(z) \varkappa^\delta(\omega) {\bm\mu}(T_z \omega) {\bm\mu}(\omega)  \, dz \, dP(\omega).
\end{array}
\end{equation}
The same change of variables as in \eqref{PropA1} results in the relation
\begin{equation}\label{Prop2_eq}
\int\limits_{\mathbb R^d} \int\limits_\Omega z a(z) \varkappa^\delta (\omega) {\bm\mu}(T_z \omega) {\bm\mu}(\omega)  \, dz \, dP(\omega)= -
\int\limits_{\mathbb R^d} \int\limits_\Omega z a(z)  \varkappa^\delta (T_z \omega)  {\bm\mu}(\omega) {\bm\mu}(T_z \omega)\, dz \, dP(\omega),
\end{equation}
therefore,  the right-hand side of \eqref{Prop2} takes the form
\begin{equation}\label{RHS}
 \!\int\limits_{\mathbb R^d}\! \int\limits_\Omega z a(z) \varkappa^\delta(\omega) {\bm\mu}(T_z \omega) {\bm\mu}(\omega)   dz  dP(\omega)= -\frac12
\int\limits_{\mathbb R^d}\! \int\limits_\Omega z a(z) \big(  \varkappa^\delta(T_z \omega) - \varkappa^\delta(\omega)   \big) {\bm\mu}(T_z \omega) {\bm\mu}(\omega)  dz dP(\omega).
\end{equation}
Equality \eqref{PropA1bis} implies that the second term on the left-hand side of \eqref{Prop2} can be rearranged in the following way
\begin{equation}\label{LHS2}
\begin{array}{c}
\displaystyle
- \int\limits_{\mathbb R^d} \int\limits_\Omega a (z) {\bm\mu} ( T_z \omega ) \big( \varkappa^\delta (T_z \omega ) - \varkappa^\delta ( \omega) \big) \varkappa^\delta(\omega){\bm\mu}(\omega) \, dz \, dP(\omega)
\\ \displaystyle
= \frac12 \int\limits_{\mathbb R^d}  \int\limits_\Omega a(z) {\bm\mu}(T_z \omega) {\bm\mu}(\omega) \big( \varkappa^\delta( T_z \omega) - \varkappa^\delta (\omega)  \big)^2  dz \, dP(\omega).
\end{array}
\end{equation}
Let us denote
$$
J^\delta = \int\limits_{\mathbb R^d}  \int\limits_\Omega {\bm\mu}(T_z \omega) {\bm\mu}(\omega) \big( \varkappa^\delta( T_z \omega) - \varkappa^\delta (\omega)  \big)^2 a(z) dz \, dP(\omega) = \int\limits_{\mathbb R^d}  \int\limits_\Omega {\bm\mu}(T_z \omega) {\bm\mu}(\omega) (u^\delta (z,\omega))^2  dM(z,\omega)
$$
and
$$
\int\limits_{\mathbb R^d}  \int\limits_\Omega \big( \varkappa^\delta( T_z \omega) - \varkappa^\delta (\omega)  \big)^2 a(z) dz \, dP(\omega) = \int\limits_{\mathbb R^d}  \int\limits_\Omega (u^\delta (z,\omega))^2  dM(z,\omega) = \| u^\delta \|^2_{L^2_M},
$$
where $dM(z, \omega) = a(z) dz dP(\omega)$.
Then
\begin{equation}\label{B1}
J^\delta = \int\limits_{\mathbb R^d}  \int\limits_\Omega {\bm\mu}(T_z \omega) {\bm\mu}(\omega) (u^\delta (z,\omega))^2  dM(z,\omega) \ge \alpha_1^2 \| u^\delta \|^2_{L^2_M}
\end{equation}
and on the other hand, relations \eqref{Prop2} - \eqref{LHS2} imply the following upper bound on $J^\delta$:
\begin{equation}\label{B2}
J^\delta = \int\limits_{\mathbb R^d}  \int\limits_\Omega {\bm\mu}(T_z \omega) {\bm\mu}(\omega) (u^\delta (z,\omega))^2  dM(z,\omega) \le \frac12 \alpha_2^2 \sigma \| u^\delta \|_{L^2_M}.
\end{equation}
Bounds \eqref{B1} - \eqref{B2} together yield
$$
\alpha_1^2 \| u^\delta \|^2_{L^2_M} \le J^\delta \le \frac12 \alpha_2^2 \sigma \| u^\delta \|_{L^2_M}.
$$
Consequently we obtain the estimate \eqref{AB} with $C = \frac{\alpha_2^2}{2 \alpha_1^2} \sigma$, and this estimate is uniform in $\delta$.
\end{proof}


\begin{corollary}
For any $\delta>0$ the following upper bound holds:
\begin{equation}\label{u-norm}
\sqrt{\delta} \, \| \varkappa^\delta \|_{L^2_\mu} \le C.
\end{equation}
\end{corollary}

\begin{proof}
From \eqref{Prop2} we have
\begin{equation}\label{Prop2-norm}
\begin{array}{c}
\displaystyle
\delta \int\limits_\Omega \big(\varkappa^\delta(\omega)\big)^2{\bm\mu}(\omega)\, dP(\omega)
  =\int\limits_{\mathbb R^d} \int\limits_\Omega a (z) {\bm\mu} ( T_z \omega ) \big( \varkappa^\delta (T_z \omega ) - \varkappa^\delta ( \omega) \big) \varkappa^\delta(\omega){\bm\mu}(\omega) \,  dz \, dP(\omega) \\ \displaystyle
    +\int\limits_{\mathbb R^d} \int\limits_\Omega z a(z) \varkappa^\delta(\omega) {\bm\mu}(T_z \omega) {\bm\mu}(\omega)  \, dz \, dP(\omega).
\end{array}
\end{equation}
Then using \eqref{RHS}, \eqref{LHS2}, \eqref{B2} together with the Cauchy-Swartz inequality and bound \eqref{AB}, we obtain that the expression on the right-hand side of \eqref{Prop2-norm} is uniformly bounded in $\delta$.
\end{proof}

Proposition \ref{boundM} implies that the family $\{ u^\delta(z, \omega) \}_{\delta>0}$ is bounded in $L^2_M$. Consequently there exists a subsequence $u_j (z, \omega) = u^{\delta_j} (z, \omega)$,  $j=1,2, \ldots,$ that converges in a weak topology of $L^2_M$ as $\delta_j \to 0$. We denote this limit by $\theta(z,\omega)$:
\begin{equation}\label{theta}
w\,\mbox{-}\!\!\lim_{j \to \infty} u_j (z,\omega) = w\,\mbox{-}\!\!\lim_{\delta_j \to 0} \big(  \varkappa^{\delta_j}(T_z \omega) - \varkappa^{\delta_j}(\omega) \big) =  \theta(z,\omega),
\end{equation}
Clearly,  $\theta(z,\omega) \in L^2_M$, i.e.
\begin{equation}\label{thetaLM}
\int\limits_{\mathbb R^d}  \int\limits_\Omega  \theta^2 (z,\omega) a(z) dz dP(\omega) < \infty,
\end{equation}
and  by the Fubini theorem $\theta (z, \omega) \in L^2 (\Omega)$ for almost all $z$ from the support of the function $a(z)$. In addition $\theta(0,\omega) \equiv 0$ and for any $z$
\begin{equation}\label{Etheta}
\mathbb{E} \theta(z,\omega) = \lim_{\delta_j \to 0} \Big( \mathbb{E} \varkappa^{\delta_j} (T_z \omega) - \mathbb{E}  \varkappa^{\delta_j}(\omega) \Big)  =  0.
\end{equation}


\medskip\noindent
{\sl Step 2.} {\sl Property A}.  The function $\theta(z,\omega)$ introduced in  \eqref{theta} is not originally defined on the set
$\{z\in\mathbb R^d\,:\,a(z)=0\}$.
\begin{proposition}\label{statincrements}
The function $\theta(z, \omega)$, given by \eqref{theta}, can be extended to  $\mathbb R^d\times\Omega$ in such a way that $\theta(z, \omega)$ satisfies relation \eqref{u-delta}, i.e. $\theta(z, \omega)$ has stationary increments:
\begin{equation}\label{thetaVIP}
 \theta(z+\xi,\omega) - \theta (\xi,\omega) = \theta(z, T_\xi \omega)  = \theta(z, T_\xi \omega) - \theta(0, T_\xi \omega).
\end{equation}
\end{proposition}

\begin{proof}
Applying Mazur's theorem \cite[Section V.1]{Yo65} we conclude that $\theta(z, \omega) = s\,\hbox{-}\!\lim\limits_{n \to \infty} w_n$ is the strong limit of a sequence $w_n$ of convex combinations of elements $u_j(z,\omega) = u^{\delta_j} (z,\omega)$.
The strong convergence implies that there exists a subsequence of $\{w_n \}$ that converges a.s. to the same limit   $\theta(z, \omega)$:
$$
\lim\limits_{n_k \to \infty} w_{n_k} (z, \omega) = \theta(z, \omega) \quad \mbox{for a.e. } \; z \; \mbox{ and a.e.} \; \omega.
$$
Since equality  \eqref{u-delta} holds for all $u_j$,  it also holds for any convex linear combination $w_n$ of $u_j$:
\begin{equation}\label{wn}
w_n (z_1 + z_2,\omega) = w_n(z_2,\omega) + w_n (z_1, T_{z_2} \omega) \quad \forall \ n.
\end{equation}
Thus taking the subsequence $\{w_{n_k} \}$ in equality \eqref{wn}  and
passing to the point-wise limit $n_k \to \infty$ in any term of this equality
we obtain \eqref{thetaVIP} first only for such $z_1, z_2$ that $z_1, z_2, z_1+ z_2$ belong to $\mathrm{supp}(a)$.
Then we extend function $\theta(z, \omega)$ to a.e. $z \in \mathbb R^d$ using relation \eqref{thetaVIP}:
\begin{equation}\label{lim_sh_inv}
\theta(z_1 + z_2, \omega) = \theta(z_2, \omega) + \theta(z_1, T_{z_2} \omega).
\end{equation}
Observe that this extension is well-defined because relation \eqref{thetaVIP} holds on the support of $a$.\\[1.5mm]
Let us show that $\theta(z,\omega)$ is defined for all $z\in\mathbb Z^d$. To this end we observe that, due to the properties
of the dynamical system $T_z$, the function $\theta(z_1,T_{z_2}\omega)$ is well-defined measurable function
of $z_1$ and $\omega$ for all $z_2\in\mathbb R^d$. The function $\theta(z_1+z_2,\omega)$ possesses the same property
due to its particular structure. Then according to \eqref{lim_sh_inv} the function $\theta(z_2, \omega)$ is defined
for all $z\in\mathbb Z^d$.
\end{proof}
Denote  $\zeta_z (\xi, \omega)= \theta(z+\xi,\omega) - \theta (\xi,\omega) $,
then for $z\in\mathbb R^d$ relation \eqref{thetaVIP} yeilds
\begin{equation}\label{thetaVIPbis}
\zeta_z (\xi, \omega) = \zeta_z(0, T_\xi \omega) ,
\end{equation}
i.e. for all $z\in\mathbb R^d$ the field $\zeta_z(\xi,\omega)$ is statistically homogeneous  in $\xi$, and
\begin{equation}\label{zetatheta}
 \zeta_z(0, \omega) = \theta(z, \omega).
\end{equation}
Thus by \eqref{theta}, \eqref{thetaVIP} -- \eqref{thetaVIPbis} the random function $\theta(z,\omega)$ is not stationary, but its increments $\zeta_z(\xi, \omega) = \theta (z+\xi, \omega ) - \theta (\xi, \omega)$ form a stationary field for any given $z$.

\bigskip\noindent
{\sl Step 3.} At this step we show that $\theta$ satisfies equation \eqref{korrkappa1}.\\
Let us prove now that $\theta(z,\omega)$ defined by \eqref{theta} is a solution of equation \eqref{korr1} (or \eqref{korrkappa1}).
To this end for an arbitrary function $\psi(\omega) \in L^2(\Omega)$ we multiply equality \eqref{A-delta-xi} by a function
$\psi(\omega){\bm\mu}(\omega)$ and integrate the resulting relation over $\Omega$, then
we have
\begin{equation}\label{Solution}
\begin{array}{c}
\displaystyle
\delta \int\limits_\Omega \varkappa^\delta(T_\xi \omega) \psi(\omega) {\bm\mu}(\omega)\, dP(\omega) \!=\!
  \int\limits_{\mathbb R^d} \int\limits_\Omega a (z) {\bm\mu} ( T_{\xi+z} \omega ) \big( \varkappa^\delta (T_{\xi+z} \omega ) - \varkappa^\delta (T_\xi \omega) \big) dz  \psi(\omega) {\bm\mu}(\omega) dP(\omega) \\ \displaystyle
    +\int\limits_{\mathbb R^d} \int\limits_\Omega z a(z) {\bm\mu}(T_{\xi+z} \omega) dz \, \psi(\omega) {\bm\mu}(\omega)  \, dP(\omega).
\end{array}
\end{equation}
By estimate \eqref{u-norm} and the Cauchy-Swartz inequality for any $\psi \in L^2(\Omega)$  we get
\begin{equation}\label{ud-norm}
\delta \int\limits_\Omega \varkappa^\delta(T_\xi \omega) \psi(\omega) {\bm\mu}(\omega)\, dP(\omega) \to 0 \quad \mbox{as } \\ \delta \to 0.
\end{equation}
Passing to the limit $\delta \to 0$ in equation \eqref{Solution}
and taking into account \eqref{theta} and \eqref{ud-norm}, we obtain that for a.e. $\omega$ the function $\theta(z,T_\xi \omega)$ satisfies the equation
\begin{equation*}\label{A-delta-xibis}
\int\limits_{\mathbb R^d} a (z) {\bm\mu} ( T_{\xi+z} \omega ) \theta(z, T_\xi \omega) )  \ dz = -  \int\limits_{\mathbb R^d}  z  a (z) {\bm\mu} (T_{\xi+z} \omega ) \ dz.
\end{equation*}
Using \eqref{thetaVIP} we get after the change of variables $z \to -z$
\begin{equation}\label{theta-xi-z}
-\int\limits_{\mathbb R^d} a (z) {\bm\mu} ( T_{\xi-z} \omega ) ( \theta (\xi-z, \omega ) - \theta ( \xi, \omega) )  \ dz +  \int\limits_{\mathbb R^d}  z  a (z) {\bm\mu} (T_{\xi-z} \omega ) \ dz =0,
\end{equation}
and it is the same as \eqref{korr1}. Thus we have proved that $\theta(z,\omega)$ is a solution of \eqref{korrkappa1}.

\medskip

\noindent
{\sl Step 4}. Property B.

Assumption \eqref{add} and inequality \eqref{thetaLM} imply that
$$
c_0  \int\limits_{{\bf B}}  \int\limits_\Omega  \theta^2 (z,\omega) dz dP(\omega) <  \int\limits_{\mathbb R^d}  \int\limits_\Omega  \theta^2 (z,\omega) a(z) dz dP(\omega) < \infty,
$$
and by the Fubini theorem  we conclude that a.s.
\begin{equation}\label{L2B}
  \int\limits_{{\bf B}}  \theta^2 (z,\omega) dz < \infty.
\end{equation}
 Thus $\theta(z,\omega) \in L^2({\bf B})$ with $\| \theta (z, \omega) \|_{L^2({\bf B})} = K(\omega)$ for a.e. $\omega$, and
  ${\mathbb E} (K(\omega))^2< \infty$.


\begin{proposition} [Sublinear growing of $\eps\theta(\frac x\eps) $ in $L_{\rm loc}^2(\mathbb R^d)$] \label{1corrector}
Denote by $\varphi_\eps (z, \omega) = \eps\, \theta \big(\frac z\eps, \omega\big)$.
Then  a.s.
\begin{equation}\label{1corrsmall}
\| \varphi_\eps (\cdot, \omega) \|_{L^2(\mathcal{Q})} \ \to \ 0 \quad \mbox{ as } \; \eps \to 0
\end{equation}
for any bounded Lipschitz domain $\mathcal{Q}\subset\mathbb R^d$.
\end{proposition}

\begin{proof}
We use in the proof inequality \eqref{L2B}
and assume in what follows without loss of the generality that ${\bf B}=[0,1]^d$.

\begin{lemma}\label{LemmaC}
The family of functions  $\varphi_\eps (z, \omega) = \eps\, \theta \big(\frac z\eps, \omega\big)$ is bounded and compact in $L^2(Q)$.
\end{lemma}
\begin{proof}
Using change of variables $\frac z\eps = y$ we have
$$
\|\varphi_\eps \|^2_{L^2(Q)} = \| \eps \,  \theta \big(\frac z\eps, \omega\big) \|^2_{L^2(Q)} = \int\limits_Q \eps^2 \,  \theta^2 \big(\frac z\eps, \omega\big) dz =
\int\limits_{\eps^{-1} Q} \eps^{d+2} \, \theta^2 (y, \omega) dy
$$
$$
= \eps^{d+2} \sum\limits_{j \in \mathbb{Z}_{ Q/\eps}} \ \int\limits_{B_j}  \, \theta^2 (y, \omega) dy = \eps^{d+2} \sum\limits_{j \in \mathbb{Z}_{Q/\eps}} \ \int\limits_{B_j}  \, (\theta (y, \omega) - \theta(j,\omega) +  \theta(j,\omega))^2  dy
$$
\begin{equation}\label{L-1}
\le {2}\eps^{d+2} \sum\limits_{j \in \mathbb{Z}_{ Q/\eps}} \ \int\limits_{B_j}  (\theta (y, \omega) -
\theta(j,\omega))^2 dy \ + \   {2}\eps^{d+2} \sum\limits_{j \in \mathbb{Z}_{ Q/\eps}} \theta^2 (j,\omega) \, |B_j|.
\end{equation}
Here $j \in \mathbb{Z}^d \cap \frac1\eps Q = \mathbb{Z}_{ Q/\eps}$, $B_j=j+[0,1)^d$.
Then if $y \in B_j$, then $y = j+z, \; z \in {\bf B} = [0,1)^d$, and we can rewrite the first term on the right-hand side of  \eqref{L-1} as follows
$$
{2}\,\eps^{d+2} \sum\limits_{j \in \mathbb{Z}_{ Q/\eps}} \ \int\limits_{{\bf B}}  (\theta (j + z, \omega) - \theta(j,\omega))^2 dz =
{2}\,\eps^{d+2} \sum\limits_{j \in \mathbb{Z}_{Q/\eps}} \ \int\limits_{{\bf B}}  \theta^2 (z, T_j \omega)  dz.
$$
Using the fact that $ \theta_B(j,\omega):=\int\limits_{{\bf B}}  \theta^2 (z, T_j \omega)  dz$ is a stationary field and  $\theta(z,\omega) \in L^2({\bf B})$, by the Birkhoff
ergodic  theorem we obtain that
$$
{2}\,\eps^{d} \sum\limits_{j \in \mathbb{Z}_{Q/\eps}} \ \int\limits_{{\bf B}}  \theta^2 (z, T_j \omega)  dz \ \to \  2 |Q| \ \mathbb{E}  \int\limits_{{\bf B}}  \theta^2 (z, \omega)  dz<\infty.
$$
Consequently, the first term in \eqref{L-1} is vanishing as $\eps \to 0$:
\begin{equation}\label{L-2}
{2}\eps^{d+2} \sum\limits_{j \in \mathbb{Z}_{Q/\eps}} \ \int\limits_{{\bf B}}  \theta^2 (z, T_j \omega)  dz \ \to \ 0.
\end{equation}
Let us prove now that a.s. the second term in \eqref{L-1} is bounded. Denoting
$$
\widehat \varphi_\eps (z) =\eps \, \widehat  \theta \big(\frac z\eps, \omega\big),
$$
where $\widehat \theta$ is a piecewise constant function: $\widehat \theta \big(\frac z\eps,\omega\big) =
 \theta \big([\frac z\eps],\omega\big)  = \theta (j,\omega)$ as $z \in  \eps B_j$, the second term in \eqref{L-1} equals to
\begin{equation}\label{L-3}
{2}\,\eps^{d+2} \sum\limits_{j \in \mathbb{Z}_{Q/\eps}} \theta^2 (j,\omega) = 2 \, \|  \eps \, \widehat  \theta \big(\frac z\eps, \omega\big) \|^2_{L^2(Q)}
=2\|\widehat \varphi_\eps(z)\|^2_{L^2(Q)}.
\end{equation}
Let us estimate the difference gradient of $ \widehat \varphi_\eps$:
$$
\| {\rm grad} \, \widehat \varphi_\eps\|^2_{(L^2(Q))^d} = \eps^2 \int\limits_Q \sum_{k=1}^d \frac{\big(
\theta\big([\frac1\eps(z+\eps e_k)], \omega\big) - \theta\big([\frac z\eps],\omega\big)  \big)^2}{\eps^2} \, dz
$$
$$
=  \int\limits_Q \sum_{k=1}^d   \big(\theta\big(\big[\frac z\eps\big] + e_k, \omega\big) -  \theta\big(\big[\frac z\eps\big],\omega\big)   \big)^2 \, dz =  \eps^d  \sum_{k=1}^d   \sum\limits_{j \in \mathbb{Z}_{Q/\eps}}  \big(\theta(j+ e_k, \omega) - \theta(j,\omega) \big)^2.
$$
But $\theta(j+ e_k, \omega) - \theta(j,\omega) = \theta(e_k, T_j \omega)$ is stationary for any given $e_k$, thus
\begin{equation}\label{L-4}
\| {\rm grad} \, \widehat \varphi_\eps\|^2_{(L^2(Q))^d} = \eps^d \sum_{k=1}^d  \sum\limits_{j \in \mathbb{Z}_{Q/\eps}}    \big(\theta(j+ e_k, \omega) - \theta(j,\omega) \big)^2 \ \to \ |Q| \sum_{k=1}^d C_k,
\end{equation}
where $C_k = \mathbb{E} \theta^2 (e_k, \omega)$.

Next we prove that a.s. the following estimate holds:
\begin{equation}\label{L-5}
\bar \theta_\eps (\omega) = \int\limits_Q \widehat \varphi_\eps (z, \omega) dz =
\eps^d  \sum\limits_{j \in \mathbb{Z}_{ Q/\eps}} \eps \, \theta(j,\omega) \le \widetilde C(\omega).
\end{equation}
We apply the induction and start with $d=1$. Using stationarity of $\theta(j+1,\omega) - \theta(j,\omega)$ we have by the ergodic theorem
$$
\eps^2 \, \Big|  \sum\limits_{j \in \mathbb{Z}_{Q/\eps}} \theta(j,\omega) \Big| \le \eps^2 \,
\sum\limits_{j \in \mathbb{Z}_{Q/\eps}} \sum_{k=0}^{j-1} |\theta(k+1,\omega) - \theta(k,\omega) |
$$
$$
\le \eps^2 \, \sum\limits_{j \in \mathbb{Z}_{Q/\eps}} \sum\limits_{k \in \mathbb{Z}_{Q/\eps}} |\theta(k+1,\omega) - \theta(k,\omega) | = \eps^2\frac{|Q|}\eps \sum\limits_{k \in \mathbb{Z}_{Q/\eps}} |\theta(e_1, T_k\omega) | \ \to \ |Q|^2 \mathbb{E} |\theta (e_1, \omega)| = \bar C_1.
$$
Thus
$$
\overline{\lim\limits_{\eps \to 0}}\ \eps^2 \, \Big|  \sum\limits_{j \in \mathbb{Z}_{Q/\eps}} \theta(j,\omega) \Big| \le \bar C_1,
$$
and this implies that for a.e. $\omega$
\begin{equation}\label{L-5A}
\sup_\eps \Big| \eps^2 \, \sum\limits_{j \in \mathbb{Z}_{Q/\eps}} \theta(j,\omega) \Big| \le \widetilde C_1(\omega),
\end{equation}
where the constant $\widetilde C_1 (\omega)$ depends only on $\omega$.

Let us show how to derive  the required upper bound in the dimension $d=2$ using \eqref{L-5A}. In this case $j~\in~\mathbb{Z}_{Q/\eps}, \ j=(j_1, j_2)$, and we assume without loss of generality  that $Q \subset [-q, q]^2$. Then
$$
\theta ((j_1, j_2), \omega) = \sum_{k=0}^{j_2 -1} \big( \theta ((j_1, k+1), \omega) - \theta ((j_1, k), \omega) \big) \ + \ \theta ((j_1, 0), \omega),
$$
and for any  $j=(j_1, j_2) \in \mathbb{Z}_{Q/\eps}$ we get
$$
| \theta ((j_1, j_2), \omega)| \le \sum_{k= - q/\eps}^{q/\eps} \big| \theta ((j_1, k+1), \omega) - \theta ((j_1, k), \omega) \big| \ + \ |\theta ((j_1, 0), \omega)|.
$$
Using \eqref{L-5A} and the ergodic property of the field $| \theta (e_2, T_j\omega)|$ we obtain the following upper bound
$$
\eps^3 \, \Big| \sum\limits_{(j_1, j_2) \in \mathbb{Z}_{Q/\eps}} \theta ((j_1, j_2), \omega) \Big| \le \eps^3  \sum_{j_1= -  q/\eps}^{q/\eps}  \frac{2q}\eps \sum_{k= - q/\eps}^{q/\eps} | \theta (e_2, T_{(j_1, k)} \omega)| \ + \ \eps^3
\sum_{j_1=- q/\eps}^{q/\eps} \frac{2q}\eps |\theta ((j_1, 0), \omega)|
$$
$$
=  2q\eps^2   \sum\limits_{(j_1, k) \in \mathbb{Z}_{Q/\eps}} | \theta (e_2, T_{(j_1, k)} \omega)| + 2q\eps^2
\sum_{j_1=- q/\eps}^{q/\eps}  |\theta ((j_1, 0), \omega)| \le  \widetilde C_2(\omega) + 2q \widetilde C_1(\omega),
$$
where $2q$ is the 1-d volume of slices of $Q$ that are orthogonal to $e_1$.
The case of $d>2$ is considered in the same way.

\medskip

Applying the standard discrete Poincar\'e inequality or the Poincar\'e inequality for piece-wise linear approximations of discrete
functions we obtain from \eqref{L-4} - \eqref{L-5} that a.s.
\begin{equation}\label{L-6}
 \|  \widehat \varphi_\eps  \|^2_{L^2(Q)} \le g_1    \Big(\int\limits_Q \widehat \varphi_\eps (z, \omega) dz \Big)^2 + g_2
 \| {\rm grad} \, \widehat \varphi_\eps\|^2_{(L^2(Q))^d} \le K(\omega),
\end{equation}
where  the constants $g_1, \; g_2$, and   $K(\omega)$ do not depend on $n$.

Thus using the same piece-wise linear approximations and considering the compactness of embedding of $H^1(Q)$ to $L^2(Q)$ we derive from \eqref{L-4} and \eqref{L-6} that
 the set of functions $\{ \widehat \varphi_\eps \}$ is compact in $L^2(Q)$. As follows from  \eqref{L-1} -- \eqref{L-2}
$$
\varphi_\eps = \widehat \varphi_\eps + \breve{ \varphi}_\eps, \quad \mbox{where } \; \breve{ \varphi}_\eps(x) =
\eps \big(\theta\big(\frac x\eps\big) - \widehat \theta\big(\frac x\eps\big)\big), \quad \| \breve{ \varphi_\eps} \|_{L^2(Q)} \to 0 \; (\eps \to 0).
$$
This together with compactness of $\{ \widehat \varphi_\eps \}$ implies the compactness of the family $\{ \varphi_\eps \}$. Lemma is proved.
\end{proof}

Next we show that any limit point of the  family $\{\varphi_\eps\}$ as $\eps\to0$ is a constant function.

\begin{lemma}\label{Prop_constfun}
Let  $\{ \varphi_\eps \}$ converge for a subsequence to $\varphi_0$ in  $L^2(Q)$. Then
$\varphi_0=const$.
\end{lemma}

\begin{proof}
According to \cite{LadSol} the set $\{\mathrm{div}\phi\,:\,\phi\in (C_0^\infty(Q))^d\}$ is dense in  the subspace
of functions from $L^2(Q)$ with zero average. It suffice to show that
\begin{equation}\label{ortog_con}
\int\limits_Q \mathrm{div}\phi(x) \varphi_\eps(x)\,dx\longrightarrow 0, \ \ \hbox{as }\eps\to0,
\end{equation}
for any $\phi=(\phi^1,\,\phi^2,\ldots,\phi^d)\in (C_0^\infty(Q))^d$. Clearly,
$$
\frac 1\eps(\phi^j(x+\eps e_j)-\phi^j(x))=\partial_{x_j}\phi^j(x)+\eps\upsilon_\eps,
$$
where $\|\upsilon_\eps\|_{L^\infty(Q)}\leq C$. Then, for sufficiently small $\eps$, we have
$$
\int\limits_Q \mathrm{div}\phi(x) \varphi_\eps(x)\,dx=\int\limits_Q (\phi^j(x+\eps e_j)-\phi^j(x))
 \theta\big(\frac x\eps,\omega\big)\,dx\,+\,o(1)
$$
$$
=\int\limits_Q \phi^j(x)\big(\theta\big(\frac x\eps-e_j,\omega\big)-\theta\big(\frac x\eps,\omega\big)\big)\,dx\,+\,o(1),
$$
where $o(1)$ tends to zero as $\eps\to0$ by Lemma \ref{LemmaC}. Since $\theta(z-e_j,\omega)-(\theta(z,\omega)$ is a stationary functions,
by the Birkhoff ergodic theorem the integral on the right-hand side converges to zero a.s. as $\eps\to 0$, and the desired statement follows.
\end{proof}

Our next goal is to show that
almost surely the limit relation in  \eqref{1corrsmall} holds.
By Lemma \ref{LemmaC} the constants $\eps c^\eps$ with $c^\eps$ defined in \eqref{hi} are a.s. uniformly in $\eps$ bounded, that is
\begin{equation}\label{co_bou}
|\eps c^\eps|\leq K(\omega)
\end{equation}
for all sufficiently small $\eps>0$.\\
Consider a convergent subsequence $\{\varphi_{\eps_n}\}_{n=1}^\infty$.
By Lemma \ref{Prop_constfun} the limit function is a constant,
denote this constant by $\varphi_0$. Assume that $\varphi_0\not=0$. Then
$$
\varphi_{\eps_n}(z)=\varphi_0+\rho_{\eps_n}(z),
$$
where $\|\rho_{\eps_n}\|_{L^2({Q})}\to0$ as $\eps_n\to0$.  Clearly, we have
$$
\varphi_{2\eps_n}(z)=2\eps_n\theta\Big(\frac z{2\eps_n}\Big)=2\eps_n\theta\Big(\frac{z/2}{\eps_n}\Big)
=2\varphi_0+2\rho_{\eps_n}\Big(\frac{z}{2}\Big)\to 2\varphi_0,
$$
because $\|\rho_{\eps_n}(\cdot/2)\|_{L^2({Q})}\to 0$ as $\eps_n\to0$.  Similarly, for any $M\in \mathbb Z^+$
we have
$$
\varphi\big._{M\eps_n}(z)\,\to\, M\varphi_0 \qquad \hbox{in }L^2({Q}).
$$
Choosing $M$ in such a way that $M|\varphi_0|> K(\omega)$ we arrive at a contradiction with \eqref{co_bou}.
Therefore, $\varphi_0=0$ for any convergent subsequence.
This yields the desired convergence
in \eqref{1corrsmall} and completes the proof of Proposition \ref{1corrector}.
%
\end{proof}
\noindent
{\sl Step 5}. Uniqueness of $\theta$.


\begin{proposition}
 [Uniqueness]\label{uniqueness}
Problem \eqref{korrkappa1} has a unique up to an additive constant  solution  $\theta(z,\omega)$, $\theta \in L^2_M$,
with statistically homogeneous increments
such that \eqref{1corrsmall} holds true.
\end{proposition}

\begin{proof}
Consider two arbitrary solutions $\theta_1(z,\omega)$ and $\theta_2(z,\omega)$ of problem \eqref{korrkappa1}.
Then the difference $\Delta (z,\omega)=\theta_1(z,\omega)-\theta_2(z,\omega)$ satisfies the equation
\begin{equation}\label{1A}
\int\limits_{\mathbb R^d} a (z) \mu ( \xi+ z, \omega ) \big(\Delta (\xi+z,\omega ) - \Delta(\xi, \omega) \big)    \ dz =0
\end{equation}
for a.e. $\omega$ and for all $\xi \in \mathbb R^d$.

Let us remark that the function $\Delta (z,\omega)$ inherits properties {\bf A)} and {\bf B)} of $\theta_1(z,\omega)$ and
$\theta_2(z,\omega)$.
Consider  a cut-off function $ \varphi (\frac{|\xi|}{R})$ parameterized by $R>0$, where  $\varphi(r)$, $r\in\mathbb R$, is a  function
defined by
$$
\varphi(r) = \left\{
\begin{array}{c}
1, \quad r \le 1, \\ 2 - r, \quad 1<r<2, \\  0, \quad r \ge 2.
\end{array}
\right.
$$
For any $R>0$, multiplying equation \eqref{1A} by $\mu(\xi, \omega) \Delta (\xi, \omega ) \varphi (\frac{|\xi|}{R})$ and  integrating
the resulting relation in $\xi$ over $ \mathbb R^d$, we obtain the following equality
\begin{equation}\label{1B}
\int\limits_{\mathbb R^d} \int\limits_{\mathbb R^d} a (z) \mu ( \xi+ z, \omega )  \mu (\xi, \omega ) \big(\Delta (\xi+z,\omega ) - \Delta(\xi, \omega) \big) \Delta(\xi, \omega)   \varphi (\frac{|\xi|}{R}) \, dz \, d \xi  =0.
\end{equation}
Using the relation $a(-z)=a(z)$, after  change of variables $z \to -z, \ \xi - z = \xi'$, we get
\begin{equation}\label{2B}
\int\limits_{\mathbb R^d} \int\limits_{\mathbb R^d} a (z) \mu ( \xi'+ z, \omega )  \mu (\xi', \omega ) \big(\Delta (\xi',\omega ) - \Delta(\xi'+z, \omega) \big) \Delta(\xi'+z, \omega)   \varphi (\frac{|\xi'+z|}{R}) \, dz \, d \xi'  =0.
\end{equation}
Renaming $\xi'$ back to $\xi$ in the last equation and taking the sum of \eqref{1B} and \eqref{2B} we obtain
$$
\int\limits_{\mathbb R^d} \int\limits_{\mathbb R^d} a (z) \mu ( \xi+ z, \omega )  \mu (\xi, \omega ) \big(\Delta (\xi+z,\omega ) - \Delta(\xi, \omega) \big) \Big( \Delta(\xi+z, \omega)   \varphi (\frac{|\xi+z|}{R}) -  \Delta(\xi, \omega)   \varphi (\frac{|\xi|}{R}) \Big)  dz \, d \xi
$$
$$
= \int\limits_{\mathbb R^d} \int\limits_{\mathbb R^d} a (z) \mu ( \xi+ z, \omega )  \mu (\xi, \omega ) \Big(\Delta (\xi+z,\omega ) - \Delta(\xi, \omega) \Big)^2  \varphi (\frac{|\xi|}{R}) \, dz \, d \xi
$$
$$
+ \int\limits_{\mathbb R^d} \int\limits_{\mathbb R^d} a (z) \mu ( \xi+ z, \omega )  \mu (\xi, \omega ) \big(\Delta (\xi+z,\omega ) - \Delta(\xi, \omega) \big) \Delta(\xi+z, \omega)  \big( \varphi (\frac{|\xi+z|}{R}) -   \varphi (\frac{|\xi|}{R}) \big)  dz \, d \xi
$$
\begin{equation}\label{2C}
= J_1^R \ + \ J_2^R = 0.
\end{equation}
Letting $R=\eps^{-1}$, we first estimate the contribution of $J_2^R $.
\begin{lemma}\label{J2} The following limit relation holds a.s.:
\begin{equation}\label{3A}
\frac{1}{R^d} |J_2^R| \ \to \ 0 \quad \mbox{ as } \; R \to \infty.
\end{equation}
\end{lemma}

\begin{proof}
Denote $\Delta_z (T_\xi  \omega ) =  \Delta (\xi+z,\omega ) - \Delta(\xi, \omega)$, then  $\Delta_z (T_\xi  \omega ) $ is stationary in $\xi$ for any given $z$.

We consider separately the integration over $|\xi| > 3R$ and $|\xi| \le 3R$ in the integral $J_2^R$:
$$
J_2^R =  \int\limits_{\mathbb R^d} \int\limits_{|\xi|>3R} a (z) \mu ( \xi+ z, \omega )  \mu (\xi, \omega ) \Delta_z(T_\xi \omega) \Delta(\xi+z, \omega)  \big( \varphi (\frac{|\xi+z|}{R}) -   \varphi (\frac{|\xi|}{R}) \big)  dz \, d \xi
$$
$$
+ \int\limits_{\mathbb R^d} \int\limits_{|\xi|\le 3R} a (z) \mu ( \xi+ z, \omega )  \mu (\xi, \omega ) \Delta_z(T_\xi \omega) \Delta(\xi+z, \omega)  \big( \varphi (\frac{|\xi+z|}{R}) -   \varphi (\frac{|\xi|}{R}) \big)  dz \, d \xi.
$$
If $|\xi| > 3R$, then $\varphi (\frac{|\xi|}{R}) = 0$. Also, $\varphi (\frac{|\xi+z|}{R})=0$ if $|\xi| > 3R$ and $|z|>R$.
Then we obtain the following upper bound
$$
\frac{1}{R^d} \int\limits_{\mathbb R^d} \int\limits_{|\xi|> 3R} a (z) \mu ( \xi+ z, \omega )  \mu (\xi, \omega ) |\Delta_z (T_\xi \omega) | |\Delta(\xi+z, \omega)|   \varphi (\frac{|\xi+z|}{R})  d\xi \, dz
$$
\begin{equation}\label{estimm}
\le \frac{\alpha_2^2  }{R^d} \int\limits_{|\eta|\le 2R}  \Big( \int\limits_{|z|>R}  |z| a (z) |\Delta_z (T_{\eta-z} \omega) |\, dz \Big) \frac1R |\Delta(\eta, \omega)|   \varphi (\frac{|\eta|}{R})  d\eta
\end{equation}
$$
\le \frac{\alpha_2^2  }{R^d} \int\limits_{|\eta|\le 2R} \phi (T_\eta \omega) \frac1R |\Delta(\eta, \omega)|   \varphi (\frac{|\eta|}{R}) \,  d\eta,
$$
where $\eta=\xi+z$,
$$
\phi (T_\eta \omega) = \int\limits_{\mathbb R^d}  |z| a (z) |\Delta_z (T_{\eta -z} \omega)| \, dz,
$$
and in the first inequality we have used the fact that  $1< \frac{|z|}{R}$ if $|z|>R$.
Since  $\Delta_z(\omega) \in L^2_M$, then $\phi(\omega) \in  L^2(\Omega)$.
Applying the Cauchy-Swartz inequality to the last integral in \eqref{estimm}  and recalling the relation $R=\eps^{-1}$ we have
\begin{equation}\label{5B}
\frac{\alpha_2^2  }{R^d} \int\limits_{|\eta|\le 2R} \phi (T_\eta \omega) \frac{|\Delta(\eta, \omega)|}{R}   \varphi (\frac{|\eta|}{R}) \,  d \eta \le \alpha_2^2 \Big( \frac{1}{R^d} \int\limits_{|\eta|\le 2R} \phi^2 (T_\eta \omega) d\eta \Big)^{\frac12} \Big( \frac{1}{R^d} \int\limits_{|\eta|\le 2R} \big(\frac{|\Delta(\eta, \omega)|}{R} \big)^2 d \eta\Big)^{\frac12} \to 0,
\end{equation}
as $R \to \infty$, because the first integral on the right hand side is bounded due to the stationarity of $\phi (T_\eta \omega)$, and the second integral tends to 0 due to sublinear growth of $\Delta(\eta, \omega)$, see \eqref{1corrsmall}.

If $|\xi| \le 3R$, then the corresponding part of $R^{-d} J_2^R$ can be rewritten as a sum of two terms
$$
\frac{1}{R^d}
 \int\limits_{\mathbb R^d} \int\limits_{|\xi| \le 3R } a (z) \mu ( \xi+ z, \omega )  \mu (\xi, \omega ) \Delta_z(T_\xi \omega) (\Delta(\xi+z, \omega) - \Delta(\xi, \omega)) \big( \varphi (\frac{|\xi+z|}{R}) -   \varphi (\frac{|\xi|}{R}) \big)  d\xi \, dz
$$
$$
+ \frac{1}{R^d}\int\limits_{\mathbb R^d} \int\limits_{| \xi | \le 3R} a (z) \mu ( \xi+ z, \omega )  \mu (\xi, \omega ) \Delta_z(T_\xi \omega) \Delta(\xi, \omega)  \big( \varphi (\frac{|\xi+z|}{R}) -   \varphi (\frac{|\xi|}{R}) \big)   d\xi \, dz = I_1 + I_2.
$$
We estimate $|I_1|$ and $|I_2|$ separately. Using the inequality $|\varphi( \frac{|x|}{R}) -   \varphi (\frac{|y|}{R}) | \le \frac{|x-y|}{R}$ by
the same arguments as above we get
$$
|I_2| \le \frac{\alpha_2^2}{R^d}  \int\limits_{\mathbb R^d} \int\limits_{|\xi| \le 3R} a (z) |\Delta_z(T_\xi \omega)| |\Delta(\xi, \omega)| \frac{|z|}{R}  d\xi \, d z
$$
$$
\le  \alpha_2^2 \Big( \frac{1}{R^d} \int\limits_{|\xi|\le 3R} \phi^2 (T_\xi \omega) d\xi \Big)^{\frac12} \Big( \frac{1}{R^d} \int\limits_{|\xi|\le 3R} \big(\frac{|\Delta(\xi, \omega)|}{R} \big)^2 d \xi\Big)^{\frac12} \to 0.
$$
To estimate $I_1$ we  divide the area of integration in $z$ into two parts: $|z|< \sqrt{R}$ and $|z| \ge \sqrt{R}$, and first consider the integral
$$
I_1^{(<)} = \frac{1}{R^d} \int\limits_{|z| < \sqrt{R}} \int\limits_{|\xi| \le 3R } a (z) \mu ( \xi+ z, \omega )  \mu (\xi, \omega ) \Delta_z^2(T_\xi \omega) \big( \varphi (\frac{|\xi+z|}{R}) -   \varphi (\frac{|\xi|}{R}) \big)  d\xi \, dz
$$
Since $|z|\leq\sqrt{R}$, we have $|\varphi( \frac{|\xi +z|}{R}) -   \varphi (\frac{|\xi|}{R}) | \le \frac{1}{\sqrt{R}}$. Therefore,
$$
|I_1^{(<)}| \le \alpha_2^2 \frac{1}{\sqrt{R}} \ \frac{1}{R^d} \int\limits_{|\xi| \le 3R }  \int\limits_{\mathbb{R}^d} a (z)  \Delta_z^2(T_\xi \omega)  dz \, d \xi \to 0,
$$
as $R \to \infty$;  here we have used the fact that
$$
 \frac{1}{R^d} \int\limits_{|\xi| \le 3R }  \int\limits_{\mathbb{R}^d} a (z)  \Delta_z^2(T_\xi \omega)  dz \, d \xi  \to c_0 \mathbb{E} \Big(  \int\limits_{\mathbb{R}^d} a (z)  \Delta_z^2(\omega)  dz  \Big)
$$
with a constant $c_0$ equal to the volume of a ball of radius $3$ in $\mathbb R^d$. We turn to the second integral
$$
I_1^{(>)} = \frac{1}{R^d} \int\limits_{|z| \ge \sqrt{R}} \int\limits_{|\xi| \le 3R } a (z) \mu ( \xi+ z, \omega )  \mu (\xi, \omega ) \Delta_z^2(T_\xi \omega) \big( \varphi (\frac{|\xi+z|}{R}) -   \varphi (\frac{|\xi|}{R}) \big)  d\xi \, dz.
$$
Considering the inequality  $|\varphi( \frac{|\xi +z|}{R}) -   \varphi (\frac{|\xi|}{R}) | \le 1$ we obtain
\begin{equation}\label{7A}
|I_1^{(>)}| \le \alpha_2^2 \frac{1}{R^d} \int\limits_{|\xi| \le 3R }  \int\limits_{|z| \ge \sqrt{R}} a (z)  \Delta_z^2(T_\xi \omega) \, dz \, d \xi.
\end{equation}
Denote by $\psi_{R}(\omega)$ the stationary function defined by
$$
\psi_{R}(\omega) = \int\limits_{|z| \ge \sqrt{R}} a (z)  \Delta_z^2( \omega) \, dz.
$$
Since $ \Delta_z( \omega) \in L^2_M$, then
\begin{equation}\label{5A}
\mathbb{E} \psi_{R}(\omega) \to 0 \quad \mbox{ as  } \;  R \to \infty.
\end{equation}
Moreover, function $\psi_{R}(\omega)$ is a.s. decreasing  in $R$.
Using the ergodic theorem, \eqref{7A} and \eqref{5A}, we conclude that  $ |I_1^{(>)}| $ tends to zero  as $R \to \infty$.
Thus we have proved that $|I_1| +|I_2| \to 0 $ as $R \to \infty$ a.s.  
 Together with \eqref{5B} 
 this implies \eqref{3A}.
\end{proof}
We proceed with the term  $J_1^R$ in \eqref{2C}:
$$
J_1^R = \int\limits_{\mathbb R^d} \int\limits_{\mathbb R^d} a (z) \mu ( \xi+ z, \omega )  \mu (\xi, \omega ) \Delta_z^2 (\xi,\omega )  \varphi (\frac{|\xi|}{R}) \, dz \, d \xi.
$$
Using the ergodic theorem we get as $R \to \infty$
\begin{equation}\label{6A}
\frac{1}{R^d} J_1^R =
\frac{1}{R^d} \int\limits_{\mathbb R^d} \int\limits_{\mathbb R^d} a (z) \mu ( \xi+ z, \omega )  \mu (\xi, \omega ) \Delta_z^2 (\xi,\omega )  \varphi (\frac{|\xi|}{R}) \, dz \, d \xi \to c_1 \mathbb{E}  \int\limits_{\mathbb R^d} a (z) {\bm\mu} ( T_z \omega )  {\bm\mu} (\omega ) \Delta_z^2 (\omega )dz,
\end{equation}
where $c_1=\int_{\mathbb R^d}\varphi(|\xi|)d\xi>0$.
Consequently from \eqref{2C} - \eqref{3A} it follows that
\begin{equation}\label{6B}
\frac{1}{R^d} |J_1^R| \ \to \ 0 \quad \mbox{ as } \; R \to \infty,
\end{equation}
and together with \eqref{6A} this implies that
\begin{equation}\label{6C}
\mathbb{E}  \int\limits_{\mathbb R^d} a (z) {\bm\mu}( T_z \omega )  {\bm\mu} (\omega ) \Delta_z^2 (\omega )dz = 0.
\end{equation}
Using condition \eqref{add}  we conclude from \eqref{6C} that $\Delta_z (\omega) 
\equiv 0$ for a.e. $z$ and a.e. $\omega$, and hence $\theta_1(z,\omega)=\theta_2(z,\omega)$.
Proposition is proved.
\end{proof}

${ }$\\[-0.8cm]
This completes the proof of Theorem \ref{t_corrector}.\end{proof}

\section{Additional terms of the asymptotic expansion}\label{s_addterms}

Recall that $I_0^\eps$ stands for the sum of all terms of order $\varepsilon^{0}$ in (\ref{K2_1}) and that $u_0\in C_0^\infty(\mathbb R^d)$.
 Our first goal is to determine the coefficients of the effective elliptic operator $\hat L$.
To this end we consider the following scalar product of  $I_0^\eps$ with a function $\varphi \in L^2(\mathbb R^d)$:
\begin{equation}\label{hatK2_1}
(I^\varepsilon_0, \varphi) =
 \int\limits_{\mathbb R^d}  \int\limits_{\mathbb R^d} \Big( \frac12 z\otimes z - z \otimes \theta
 \big(\frac{x}{\varepsilon}-z, \omega \big) \Big) \  a (z) \mu \big( \frac{x}{\varepsilon}, \omega \big) \mu \big( \frac{x}{\varepsilon} -z, \omega \big)  \ dz \  \nabla \nabla u_0 (x) \varphi(x) dx.
\end{equation}
After change of variables $x = \varepsilon \eta$ we have
\begin{equation}\label{hatK2_2}
\begin{array}{l}
\displaystyle
(I^\varepsilon_0, \varphi) =
\varepsilon^d \int\limits_{\mathbb R^d}  \int\limits_{\mathbb R^d} \frac12 a (z) \,z\otimes z  \, \mu ( \eta, \omega ) \mu ( \eta -z, \omega ) \,  dz \, \nabla \nabla u_0 (\varepsilon\eta) \, \varphi (\varepsilon \eta)  \,  d\eta \\ \displaystyle
- \varepsilon^d \int\limits_{\mathbb R^d}  \int\limits_{\mathbb R^d} a (z) \,z \otimes  \theta (\eta-z, \omega )  \mu ( \eta, \omega ) \mu ( \eta -z, \omega ) \, dz \,  \nabla \nabla u_0 (\varepsilon\eta) \, \varphi (\varepsilon \eta)  \,  d\eta = I^\eps_1(\varphi) - I^\eps_2(\varphi).
\end{array}
\end{equation}
We consider the integrals $I^\eps_1(\varphi)$ and $I^\eps_2(\varphi)$ separately.
Since $\int_{\mathbb R^d}|z|^2a(z)ds\leq\infty$, then
$$
\int\limits_{\mathbb R^d} z\otimes z \,a(z)  \mu (0,\omega)\mu(-z,\omega)\,dz \in (L^\infty(\Omega))^{d^2}.
$$
Therefore,  by the Birkhoff ergodic theorem a.s.
$$
\int\limits_{\mathbb R^d} z\otimes z\,a(z)  \mu (\frac{x}{\eps},\omega)\mu(\frac{x}{\eps}-z,\omega)\,dz \rightharpoonup
D_1\quad\hbox{weakly in } \ (L^2_{\rm loc}(\mathbb R^d))^{d^2}
$$
with
\begin{equation}\label{J_1}
D_1 =  \int\limits_{\mathbb R^d} \frac12 \, z\otimes z \, a (z) \,  E\{ \mu ( 0, \omega ) \mu ( -z, \omega )\} \,  dz.
\end{equation}
Recalling that $u_0\in C_0^\infty(\mathbb R^d)$, we obtain
\begin{equation}\label{I_1}
I^\eps_1(\varphi)\to  \int\limits_{\mathbb R^d}D_1\nabla\nabla u_0(x)\varphi(x)\,dx.
\end{equation}

The second integral in \eqref{hatK2_2} contains the non-stationary random field $ \theta (z,\omega)$, and we rewrite $I_2(\varphi)$ as a sum of two terms, such that the first term contains the stationary field $\zeta_z (\eta, \omega)$ and the contribution of the second one is asymptotically negligible. In order to estimate the contribution of the second term we construct an additional corrector $u_2^\varepsilon$, see formula \eqref{corr-u2} below.\\
We have
\begin{equation}\label{I_2appr}
\begin{array}{l}
\displaystyle
I^\varepsilon_2 (\varphi) =  \int\limits_{\mathbb R^d}\! \int\limits_{\mathbb R^d}   a (z) z  \,  \mu (\frac{x}{\varepsilon}, \omega ) \mu (\frac{x}{\varepsilon} -z, \omega )   \theta (\frac{x}{\varepsilon} - z, \omega ) \nabla \nabla u_0(x) \varphi(x) \, d x \, dz \\ \displaystyle
= \frac12 \int\limits_{\mathbb R^d}\! \int\limits_{\mathbb R^d}   a (z) z  \,  \mu (\frac{x}{\varepsilon}, \omega ) \mu (\frac{x}{\varepsilon} -z, \omega )   \theta (\frac{x}{\varepsilon} - z, \omega ) \nabla \nabla u_0(x) \varphi(x) \, d x \, dz \\ \displaystyle
- \, \frac12 \int\limits_{\mathbb R^d}\! \int\limits_{\mathbb R^d}   a (z) z \,   \mu (\frac{y}{\varepsilon}, \omega ) \mu (\frac{y}{\varepsilon} -z, \omega )   \theta (\frac{x}{\varepsilon} - z, \omega ) \nabla \nabla u_0(y - \varepsilon z) \varphi(y-\varepsilon z) \, d y \, dz \\ \displaystyle
= \frac12 \int\limits_{\mathbb R^d}\! \int\limits_{\mathbb R^d} \!  a (z) z \,  \mu (\frac{x}{\varepsilon}, \omega ) \mu (\frac{x}{\varepsilon} -z, \omega ) \Big(  \theta (\frac{x}{\varepsilon} - z, \omega )  \nabla \nabla u_0(x) \varphi(x) -   \theta (\frac{x}{\varepsilon}, \omega )  \nabla \nabla u_0(x - \varepsilon z) \varphi (x-\varepsilon z)\! \Big) d x  dz \\ \displaystyle
=  \frac12 \int\limits_{\mathbb R^d}\! \int\limits_{\mathbb R^d}   a (z) z  \, \mu (\frac{x}{\varepsilon}, \omega ) \mu (\frac{x}{\varepsilon} -z, \omega ) \big(  \theta (\frac{x}{\varepsilon} - z, \omega ) -  \theta (\frac{x}{\varepsilon}, \omega ) \big)  \nabla \nabla u_0(x) \varphi(x)  d x \, dz  \\ \displaystyle
+  \frac12 \int\limits_{\mathbb R^d}\! \int\limits_{\mathbb R^d}   a (z) z  \, \mu (\frac{x}{\varepsilon}, \omega ) \mu (\frac{x}{\varepsilon} -z, \omega ) \,  \theta (\frac{x}{\varepsilon}, \omega ) \big( \nabla \nabla u_0(x) \varphi(x) -   \nabla \nabla u_0(x - \varepsilon z) \varphi (x-\varepsilon z) \big) d x \, dz,
\end{array}
\end{equation}
here and in what follows $z\theta(z)\nabla\nabla u_0(x)$ stands for $z^i\theta^j(z)\partial_{x_i}\partial_{x_j}u_0(x)$.
The field $\zeta_{-z} (\eta, \omega)=  \theta(\eta -z,\omega) -  \theta (\eta,\omega)$ is stationary for any given $z$, and
\begin{equation}\label{PL1}
\int\limits_{\mathbb R^d}  a (z) z \otimes   \zeta_{-z} (0, \omega) \mu ( 0, \omega ) \mu (  -z, \omega ) \, dz \in (L^2(\Omega))^{d^2}.
\end{equation}
Indeed, in view of \eqref{thetaLM} and \eqref{zetatheta} by the Cauchy-Schwarz inequality we have
$$
\int\limits_{\Omega}\bigg(\int\limits_{\mathbb R^d} |a (z) z \otimes   \zeta_{-z} (0, \omega) \mu ( 0, \omega ) \mu (  -z, \omega )| \, dz\bigg)^2 d P(\omega) \le
$$
$$
\alpha_2^2 \Big(\int\limits_{\mathbb R^d}  a (z) |z|^2 dz \Big) \Big( \int\limits_{\mathbb R^d} \int\limits_{\Omega} a (z) \, |\theta(-z, \omega)|^2 dz d P(\omega) \Big) < \infty.
$$
Consequently applying the ergodic theorem to the stationary field \eqref{PL1} we obtain for the first integral in \eqref{I_2appr} as $\varepsilon \to 0$
\begin{equation}\label{I2-stationary}
\begin{array}{l}
\displaystyle
\frac12 \int\limits_{\mathbb R^d} \int\limits_{\mathbb R^d} \,  a (z) z   \zeta_{-z} (\frac{x}{\varepsilon}, \omega ) \mu (\frac{x}{\varepsilon}, \omega ) \mu (\frac{x}{\varepsilon} -z, \omega )   \nabla \nabla u_0(x) \varphi(x)  d x \, dz  \ \to
\\ \displaystyle
\frac12 \int\limits_{\mathbb R^d} \int\limits_{\mathbb R^d}  a (z) z   E\{  \zeta_{-z} (0, \omega) \mu ( 0, \omega ) \mu (  -z, \omega ) \}  \nabla \nabla u_0(x) \varphi(x)  d x \, dz =  \int\limits_{\mathbb R^d} D_2 \, \nabla \nabla u_0 (x) \varphi(x) \, dx,
\end{array}
\end{equation}
where we have used the notation
\begin{equation}\label{D_2}
D_2 =  \frac12 \, \int\limits_{\mathbb R^d} a (z) z  \otimes E\{ \zeta_{-z} (0, \omega) \mu ( 0, \omega ) \mu ( -z, \omega )\} \,  dz.
\end{equation}
Denote the last integral on the right-hand side  in \eqref{I_2appr} by $J_2^\varepsilon (\varphi)$:
\begin{equation}\label{J2eps}
J_2^\varepsilon (\varphi) = \frac12 \int\limits_{\mathbb R^d} \int\limits_{\mathbb R^d} \,  a (z) z  \, \mu (\frac{x}{\varepsilon}, \omega ) \mu (\frac{x}{\varepsilon} -z, \omega ) \,  \theta (\frac{x}{\varepsilon}, \omega ) \big( \nabla \nabla u_0(x) \varphi(x) -   \nabla \nabla u_0(x - \varepsilon z) \varphi (x-\varepsilon z) \big) d x \, dz
\end{equation}
and consider this expression as a functional on $L^2(\mathbb R^d)$ acting on function $\varphi$.
In order to show that for each $\eps>0$ the functional $J_2^\varepsilon$
is a bounded linear functional on $L^2(\mathbb R^d)$ we represent $J_2^\varepsilon$ as a sum $J_2^\varepsilon=J_2^{1,\varepsilon}
+J_2^{2,\varepsilon}+J_2^{3,\varepsilon}$ with  $J_2^{1,\varepsilon}$,
$J_2^{2,\varepsilon}$ and $J_2^{3,\varepsilon}$ introduced below and  estimate each of  these functionals separately. By Proposition \ref{1corrector} a.s.
$ \theta (\frac{x}{\eps},\omega)\in L^2_{\rm loc}(\mathbb R^d)$ for all $\varepsilon>0$.  Therefore,
$$
J_2^{1,\varepsilon} (\varphi) = \frac12 \int\limits_{\mathbb R^d} \int\limits_{\mathbb R^d} \,  a (z) z  \, \mu (\frac{x}{\varepsilon}, \omega ) \mu (\frac{x}{\varepsilon} -z, \omega ) \,  \theta (\frac{x}{\varepsilon}, \omega )  \nabla \nabla u_0(x) \varphi(x)  d x \, dz
$$
is a.s. a bounded linear functional on $L^2(\mathbb R^d)$. Similarly,
$$
J_2^{2,\varepsilon} (\varphi) = \frac12 \int\limits_{\mathbb R^d} \int\limits_{\mathbb R^d} \,  a (z) z  \, \mu (\frac{x}{\varepsilon}, \omega ) \mu (\frac{x}{\varepsilon} -z, \omega ) \,  \theta (\frac{x}{\varepsilon}-z, \omega )  \nabla \nabla u_0(x-\eps z)
\varphi(x-\eps z)  d x \, dz
$$
is a.s. a bounded linear functional on $L^2(\mathbb R^d)$. Due to \eqref{thetaLM} and by the Birkhoff ergodic theorem the linear functional
$$
J_2^{3,\varepsilon} (\varphi) = \frac12 \int\limits_{\mathbb R^d} \int\limits_{\mathbb R^d} \,  a (z) z  \, \mu (\frac{x}{\varepsilon}, \omega ) \mu (\frac{x}{\varepsilon} -z, \omega ) \, \Big( \theta (\frac{x}{\varepsilon}, \omega )-
 \theta (\frac{x}{\varepsilon}-z, \omega )\Big)  \nabla \nabla u_0(x-\eps z)
\varphi(x-\eps z)  d x \, dz
$$
$$
= \frac12 \int\limits_{\mathbb R^d} \int\limits_{\mathbb R^d} \,  a (z) z  \, \mu (\frac{x}{\varepsilon}+z, \omega ) \mu (\frac{x}{\varepsilon} , \omega ) \, \Big( \theta (\frac{x}{\varepsilon}+z, \omega )-
 \theta (\frac{x}{\varepsilon}, \omega )\Big)  \nabla \nabla u_0(x)
\varphi(x)  d x \, dz
$$
$$
= \frac12 \int\limits_{\mathbb R^d} \int\limits_{\mathbb R^d} \,  a (z) z  \, \mu (\frac{x}{\varepsilon}+z, \omega ) \mu (\frac{x}{\varepsilon} , \omega ) \,  \theta
 (z,T_{\frac{x}{\varepsilon}} \omega )  \nabla \nabla u_0(x)
\varphi(x)  d x \, dz
$$
is a.s. bounded in $L^2(\mathbb R^d)$. Since $J_2^{\varepsilon} (\varphi) =J_2^{1,\varepsilon} (\varphi)+ J_2^{2,\varepsilon} (\varphi)+ J_2^{3,\varepsilon} (\varphi)$, the desired boundedness of  $J_2^{\varepsilon}$ follows.
Then by the Riesz theorem for a.e. $\omega$ there exists a function $f_2^\varepsilon  = f_2^\varepsilon(u_0) \in L^2(\mathbb R^d)$ such that $J_2^\varepsilon(\varphi) = (f_2^\varepsilon,\varphi)$. We emphasize that here we do not claim that the norm of $J_2^\eps$ admits
a uniform in $\eps$ estimate.

Next we show that the contribution of $f_2^\varepsilon$ to $w^\varepsilon$ is vanishing. To this end consider the function (additional corrector)
\begin{equation}\label{corr-u2}
u_2^\varepsilon (x,\omega) = (-L^\varepsilon +m)^{-1} f_2^\varepsilon (x, \omega).
\end{equation}
\begin{lemma}\label{l_u2small}
 $\|  u_2^\varepsilon\|_{L^2(\mathbb R^d)} \to 0$ as $\varepsilon \to 0$ for a.e. $\omega$.
\end{lemma}

\begin{proof}
Taking $\varphi = u_2^\varepsilon$ we get
\begin{equation}\label{L1}
 ((-L^\varepsilon +m) u_2^\varepsilon, u_2^\varepsilon) = (f_2^\varepsilon, u_2^\varepsilon).
\end{equation}
Considering \eqref{L_eps} the left-hand side of \eqref{L1} can be rearranged as follows:
\begin{equation}\label{L1-LHS}
\begin{array}{l}
\displaystyle
- \frac{1}{\varepsilon^2} \int\limits_{\mathbb R^d} \int\limits_{\mathbb R^d} \,  a (z)   \, \mu (\frac{x}{\varepsilon}, \omega ) \mu (\frac{x}{\varepsilon} -z, \omega ) ( u_2^\varepsilon (x-\varepsilon z) -  u_2^\varepsilon(x)) dz \, u_2^\varepsilon (x) dx + m  \int\limits_{\mathbb R^d} (u_2^\varepsilon)^2 (x) dx \\ \displaystyle
= \, \frac12 \frac{1}{\varepsilon^2} \int\limits_{\mathbb R^d} \int\limits_{\mathbb R^d} \,  a (z)   \, \mu (\frac{x}{\varepsilon}, \omega ) \mu (\frac{x}{\varepsilon} -z, \omega ) ( u_2^\varepsilon (x-\varepsilon z) -  u_2^\varepsilon(x))^2 dz dx + m  \int\limits_{\mathbb R^d} (u_2^\varepsilon)^2 (x) dx.
\end{array}
\end{equation}
We denote
$$
G_1^2 = \frac{1}{2 \varepsilon^2} \int\limits_{\mathbb R^d} \int\limits_{\mathbb R^d} \,  a (z)   \, \mu (\frac{x}{\varepsilon}, \omega ) \mu (\frac{x}{\varepsilon} -z, \omega ) ( u_2^\varepsilon (x-\varepsilon z) -  u_2^\varepsilon(x))^2 dz dx, \quad G_2^2= m  \int\limits_{\mathbb R^d} (u_2^\varepsilon)^2 (x) dx.
$$
It follows from \eqref{J2eps} that the right-hand side of \eqref{L1} takes the form
\begin{equation}\label{L1-RHS}
\begin{array}{l}
\displaystyle
J_2^\varepsilon (u_2^\varepsilon) = \frac12 \int\limits_{\mathbb R^d} \int\limits_{\mathbb R^d}   a (z) z  \, \mu (\frac{x}{\varepsilon}, \omega ) \mu (\frac{x}{\varepsilon} -z, \omega ) \,  \theta (\frac{x}{\varepsilon}, \omega ) \big( \nabla \nabla u_0(x)  u_2^\varepsilon(x) -   \nabla \nabla u_0(x - \varepsilon z)  u_2^\varepsilon (x-\varepsilon z) \big) d x \, dz \\ \displaystyle
=  \frac12 \, \int\limits_{\mathbb R^d} \int\limits_{\mathbb R^d} \,  a (z) z  \, \mu (\frac{x}{\varepsilon}, \omega ) \mu (\frac{x}{\varepsilon} -z, \omega ) \,  \theta (\frac{x}{\varepsilon}, \omega )  \nabla \nabla u_0(x) \big( u_2^\varepsilon(x) -   u_2^\varepsilon (x-\varepsilon z) \big) d x \, dz \\[6mm] \displaystyle
+ \frac12 \int\limits_{\mathbb R^d} \int\limits_{\mathbb R^d}  a (z) z  \, \mu (\frac{x}{\varepsilon}, \omega ) \mu (\frac{x}{\varepsilon} -z, \omega ) \,  \theta (\frac{x}{\varepsilon}, \omega ) \big( \nabla \nabla u_0(x)  -   \nabla \nabla u_0(x - \varepsilon z) \big) u_2^\varepsilon (x-\varepsilon z) d x \, dz =\! \frac12 (I_1 + I_2).
\end{array}
\end{equation}
It is proved in Proposition \ref{1corrector} that a.s. $\|\eps  \theta (\frac x\eps,\omega)\|_{L^2(B)}\to 0$ as $\eps\to0$ for any
ball $B\subset\mathbb R^d$.
By the Cauchy-Schwartz inequality  we obtain the following upper bounds for $I_1$:
\begin{equation}\label{L1-RHS-I1}
\begin{array}{l}
\displaystyle
I_1 \le \left( \int\limits_{\mathbb R^d} \int\limits_{\mathbb R^d} \,  a (z) \mu (\frac{x}{\varepsilon}, \omega ) \mu (\frac{x}{\varepsilon} -z, \omega ) \big( u_2^\varepsilon(x) -   u_2^\varepsilon (x-\varepsilon z) \big)^2 d x \, dz \right)^{1/2} \\
\displaystyle
\left( \frac{1}{\varepsilon^2 } \, \int\limits_{\mathbb R^d} \int\limits_{\mathbb R^d} \,  a (z)|z|^2 \, \mu (\frac{x}{\varepsilon}, \omega ) \mu (\frac{x}{\varepsilon} -z, \omega ) \, \varepsilon^2   \big|\theta (\frac{x}{\varepsilon}, \omega )\big|^2  (\nabla \nabla u_0(x))^2 d x \, dz \right)^{1/2} \\
\displaystyle
\le \frac{1}{\varepsilon} \, o(1) \ \left(\frac12 \int\limits_{\mathbb R^d} \int\limits_{\mathbb R^d} \,  a (z) \mu (\frac{x}{\varepsilon}, \omega ) \mu (\frac{x}{\varepsilon} -z, \omega ) \big( u_2^\varepsilon(x) -   u_2^\varepsilon (x-\varepsilon z) \big)^2 d x \, dz \right)^{1/2} = G_1 \cdot o(1),
\end{array}
\end{equation}
where $o(1)\to0$ as $\eps\to0$. We turn to the second integral $I_2$. Let $B$ be a ball centered at the origin and such that
$\mathrm{supp}(u_0)\subset B$, $\mathrm{dist}(\mathrm{supp}(u_0),\partial B)>1$. Then
$$
\Big|\int\limits_{\mathbb R^d} \int\limits_{B} \,  a (z) z  \, \mu (\frac{x}{\varepsilon}, \omega ) \mu (\frac{x}{\varepsilon} -z, \omega ) \,  \theta (\frac{x}{\varepsilon}, \omega ) \big( \nabla \nabla u_0(x)  -   \nabla \nabla u_0(x - \varepsilon z) \big) u_2^\varepsilon (x-\varepsilon z) d x \, dz\Big|
$$
\begin{equation}\label{aaa1}
\leq C\int\limits_{\mathbb R^d} \int\limits_{B} \,  a (z) |z|^2  \, \big|\eps  \theta (\frac{x}{\varepsilon}, \omega )\big|\,
| u_2^\varepsilon (x-\varepsilon z)| d x \, dz\le \| u_2^\varepsilon \|_{L^2(\mathbb R^d)} \cdot o(1) = G_2 \cdot o(1).
\end{equation}
The integral over $B^c=\mathbb R^d\setminus B$ can be estimated in the following way:
$$
\Big|\int\limits_{\mathbb R^d} \int\limits_{B^c} \,  a (z) z  \, \mu (\frac{x}{\varepsilon}, \omega ) \mu (\frac{x}{\varepsilon} -z, \omega ) \,  \theta (\frac{x}{\varepsilon}, \omega ) \big( \nabla \nabla u_0(x)  -   \nabla \nabla u_0(x - \varepsilon z) \big) u_2^\varepsilon (x-\varepsilon z) d x \, dz\Big|
$$
$$
\Big|\int\limits_{\mathbb R^d} \int\limits_{B^c} \,  a (z) z  \, \mu (\frac{x}{\varepsilon}, \omega ) \mu (\frac{x}{\varepsilon} -z, \omega ) \,  \theta (\frac{x}{\varepsilon}, \omega )   \nabla \nabla u_0(x - \varepsilon z)  u_2^\varepsilon (x-\varepsilon z) d x \, dz\Big|
$$
\begin{equation}\label{aaa2}
\leq C\int\limits_{|z|\geq \frac1\eps} \int\limits_{B^c} \,  a (z) |z|  \, \big| \theta (\frac{x}{\varepsilon}, \omega )\big|\,
|\nabla \nabla u_0(x - \varepsilon z)|\,  |u_2^\varepsilon (x-\varepsilon z)|\, d x \, dz
\end{equation}
$$
\leq C\int\limits_{|z|\geq \frac1\eps} \int\limits_{\mathbb R^d} \,  a (z) |z|  \, \big| \theta (\frac{x}{\varepsilon}+z, \omega )\big|\,
|\nabla \nabla u_0(x)|\,  |u_2^\varepsilon (x)|\, d x \, dz
$$
$$
\leq C\int\limits_{|z|\geq \frac1\eps} \int\limits_{\mathbb R^d} \,  a (z) |z|  \,\Big[ \big| \theta (\frac{x}{\varepsilon}+z, \omega )
-  \theta (\frac{x}{\varepsilon}, \omega )\big|+\big| \theta (\frac{x}{\varepsilon}, \omega )\big|\Big]\,
|\nabla \nabla u_0(x)|\,  |u_2^\varepsilon (x)|\, d x \, dz.
$$
We have
$$
\int\limits_{|z|\geq \frac1\eps} \int\limits_{\mathbb R^d} \,  a (z) |z|  \,\big| \theta (\frac{x}{\varepsilon}, \omega )\big|\,
|\nabla \nabla u_0(x)|\,  |u_2^\varepsilon (x)|\, d x \, dz
$$
$$
\leq
\int\limits_{\mathbb R^d} \int\limits_{\mathbb R^d} \,  a (z) |z|^2  \,\big|\eps  \theta (\frac{x}{\varepsilon}, \omega )\big|\,
|\nabla \nabla u_0(x)|\,  |u_2^\varepsilon (x)|\, d x \, dz
 \leq G_2\cdot o(1)
$$
and
$$
\int\limits_{|z|\geq \frac1\eps} \int\limits_{\mathbb R^d} \,  a (z) |z|  \,\Big[ \big| \theta (\frac{x}{\varepsilon}+z, \omega )
-  \theta (\frac{x}{\varepsilon}, \omega )\big|\Big]\,
|\nabla \nabla u_0(x)|\,  |u_2^\varepsilon (x)|\, d x \, dz
$$
$$
\leq \int\limits_{|z|\geq \frac1\eps} \int\limits_{\mathbb R^d} \,  a (z) |z|  \, \big| \zeta_z (T_{\frac{x}{\varepsilon}}\omega )
\big|\,
|\nabla \nabla u_0(x)|\,  |u_2^\varepsilon (x)|\, d x \, dz
$$
$$
\leq \left( \int\limits_{|z|\geq \frac1\eps} \,  a (z) z^2  \, dz \right)^{\frac12}  \int\limits_{\mathbb R^d} \left( \int\limits_{\mathbb R^d} a(z) \big| \zeta_z (T_{\frac{x}{\varepsilon}}\omega )
\big|^2 \, dz \right)^{\frac12}
|\nabla \nabla u_0(x)|\,  |u_2^\varepsilon (x)|\, d x
$$
$$
\leq o(1) \, \left( \int\limits_{\mathbb R^d} \,  |u_2^\varepsilon (x)|^2 \, dx \right)^{\frac12} \left( \int\limits_{\mathbb R^d} \left( \int\limits_{\mathbb R^d} a(z) \big| \zeta_z (T_{\frac{x}{\varepsilon}}\omega )
\big|^2 \, dz \right)
|\nabla \nabla u_0(x)|^2\, d x \right)^{\frac12}  = G_2\cdot o(1).
$$
Since $\zeta_z (\omega) \in L^2_M$, the second integral in the right hand side here converges to a constant by the ergodic theorem.

Combining the last two estimates we conclude that the term on the right-hand side in \eqref{aaa2} does not exceed $G_2\cdot o(1)$.
Therefore, considering \eqref{aaa1}, we obtain $I_1\leq G_2\cdot o(1)$. This estimate and \eqref{L1-RHS-I1} imply that
$$
G_1^2 + G_2^2 = I_1 + I_2 \le (G_1 + G_2) \cdot o(1).
$$
Consequently, $G_1 \to 0$ and $G_2 = m^{1/2} \| u_2^\varepsilon \|_{L^2(\mathbb R^d)} \to 0$ as $\varepsilon \to 0$. Lemma is proved.
\end{proof}

\medskip

Thus we can rewrite  $I^\varepsilon_0$ (all the terms of the order $\varepsilon^{0}$) as follows
\begin{equation}\label{VV}
I^\varepsilon_0 = (D_1 - D_2) \cdot \nabla\nabla u_0 +  f_2^\varepsilon + S(\frac{x}{\varepsilon}, \omega) \cdot \nabla\nabla u_0, \qquad  S(\frac{x}{\varepsilon}, \omega) = \Psi_1(\frac{x}{\varepsilon}, \omega) - \Psi_2(\frac{x}{\varepsilon}, \omega),
\end{equation}
where the matrices $D_1$and  $D_2$ are defined in \eqref{J_1} and \eqref{D_2} respectively, and $ S(\frac{x}{\varepsilon}, \omega), \Psi_1(\frac{x}{\varepsilon}, \omega), \Psi_2(\frac{x}{\varepsilon}, \omega)$ are stationary fields with zero mean which are given by
\begin{equation}\label{Psi-1}
\Psi_1(\frac{x}{\varepsilon}, \omega) = \frac12 \int\limits_{\mathbb R^d} \,  a (z) z^2   \Big[  \mu (\frac{x}{\varepsilon}, \omega ) \mu (\frac{x}{\varepsilon} -z, \omega )  -
 E\{  \mu ( 0, \omega ) \mu (  -z, \omega ) \} \Big] dz,
\end{equation}
\begin{equation}\label{Psi-2}
\Psi_2(\frac{x}{\varepsilon}, \omega) = \frac12 \int\limits_{\mathbb R^d} \,  a (z) z  \Big[ \zeta_{-z} (\frac{x}{\varepsilon}, \omega ) \mu (\frac{x}{\varepsilon}, \omega ) \mu (\frac{x}{\varepsilon} -z, \omega )  -
 E\{ \zeta_{-z} (0, \omega) \mu ( 0, \omega ) \mu (  -z, \omega )\} \Big] dz.
\end{equation}
Denote
\begin{equation}\label{u3}
u_3^\varepsilon(x,\omega) = (-L^\varepsilon+m)^{-1} F^\varepsilon(x,\omega), \quad \mbox{where } \; F^\varepsilon(x, \omega) = S(\frac{x}{\varepsilon}, \omega) \cdot \nabla\nabla u_0(x).
\end{equation}
Since $ {\rm supp} \, u_0 \subset B$ is a bounded subset of $\mathbb R^d$ and
$$
 \int\limits_{\mathbb R^d} \,  a (z) |z|\, \big|\zeta_{-z} ( \omega )\big|  \,dz \in L^2(\Omega),
$$
then by the Birkhoff theorem $u_3^\varepsilon \in L^2(\mathbb R^d)$. Our goal is to prove that $\|u_3^\varepsilon \|_{L^2(\mathbb R^d)} \to 0$ as $\varepsilon \to 0$. We first show that the family $\{u_3^\varepsilon\}$ is bounded in $L^2(\mathbb R^d)$.
\begin{lemma}\label{Bound}
The family of functions $u_3^\varepsilon$ defined by \eqref{u3} is uniformly bounded in $L^2(\mathbb R^d)$ for e.a. $\omega$: $\|u_3^\varepsilon \|_{L^2(\mathbb R^d)} \le C$ for any $0<\varepsilon<1$.
\end{lemma}

\begin{proof}
Since the operator $(-L^\varepsilon+m)^{-1} $ is bounded ($\| (-L^\varepsilon+m)^{-1} \| \le \frac{1}{m}$), then it is sufficient to prove that $\| F^\varepsilon(x,\omega) \|_{L^2(\mathbb R^d)} \le C$ uniformly in $\varepsilon$. By the Birkhoff ergodic theorem the functions $ \Psi_1(\frac{x}{\varepsilon}, \omega)$ and $\Psi_2(\frac{x}{\varepsilon}, \omega)$ a.s converge
to zero weakly in $L^2(B)$, so does $S(\frac{x}{\varepsilon}, \omega)$. Then $S(\frac{x}{\varepsilon}, \omega)\cdot \nabla\nabla
u_0$ a.s. converges to zero weakly in $L^2(\mathbb R^d)$.
This implies the desired boundedness.
\end{proof}


\begin{lemma}\label{Convergence} For any cube $B$ centered at the origin
 $\|u_3^\varepsilon \|_{L^2(B)} \ \to \ 0$ as $\varepsilon \to 0$ for e.a. $\omega$.
\end{lemma}

\begin{proof}
The first step of the proof is to show that any sequence $\{u_3^{\varepsilon_j} \}$,  $\varepsilon_j \to 0$, is compact in $L^2(B)$.
Using definition \eqref{u3} we have
$$
( (-L^\varepsilon+m) u_3^\varepsilon, u_3^\varepsilon) \ = \  (  F^\varepsilon, u_3^\varepsilon).
$$
The left-hand side of this relation can be rewritten as
\begin{equation}\label{L2-rhs}
\begin{array}{l}
\displaystyle
\int\limits_{\mathbb R^d} (-L^\varepsilon+m) u_3^\varepsilon(x)  u_3^\varepsilon(x) dx  \\ \displaystyle
= \, m  \int\limits_{\mathbb R^d} (u_3^\varepsilon(x))^2 dx  - \frac{1}{\varepsilon^2} \int\limits_{\mathbb R^d} \int\limits_{\mathbb R^d} \,  a (z)   \, \mu (\frac{x}{\varepsilon}, \omega ) \mu (\frac{x}{\varepsilon} -z, \omega) ( u_3^\varepsilon (x-\varepsilon z) -  u_3^\varepsilon(x)) u_3^\varepsilon (x) dz dx  \\ \displaystyle
= \, m  \int\limits_{\mathbb R^d} (u_3^\varepsilon(x))^2 dx  + \frac{1}{2 \varepsilon^2} \int\limits_{\mathbb R^d} \int\limits_{\mathbb R^d} \,  a (z)   \, \mu (\frac{x}{\varepsilon}, \omega ) \mu (\frac{x}{\varepsilon} -z, \omega) ( u_3^\varepsilon (x-\varepsilon z) -  u_3^\varepsilon(x))^2 dz dx.
\end{array}
\end{equation}
Consequently we obtain the following equality
\begin{equation}\label{u3-main}
 m  \int\limits_{\mathbb R^d} (u_3^\varepsilon(x))^2 dx  + \frac{1}{2 \varepsilon^2} \int\limits_{\mathbb R^d} \int\limits_{\mathbb R^d} \,  a (z)   \, \mu (\frac{x}{\varepsilon}, \omega ) \mu (\frac{x}{\varepsilon} -z, \omega) ( u_3^\varepsilon (x-\varepsilon z) -  u_3^\varepsilon(x))^2 dz dx = (  F^\varepsilon, u_3^\varepsilon).
\end{equation}
Considering the uniform boundedness of $F^\varepsilon$ and $ u_3^\varepsilon$, see Lemma \ref{Bound}, we immediately conclude that
\begin{equation}\label{C-main}
\frac{1}{\varepsilon^2} \int\limits_{\mathbb R^d} \int\limits_{\mathbb R^d} \,  a (z)   \, \mu (\frac{x}{\varepsilon}, \omega ) \mu (\frac{x}{\varepsilon} -z, \omega) ( u_3^\varepsilon (x-\varepsilon z) -  u_3^\varepsilon(x))^2 dz dx < K
\end{equation}
uniformly in $\varepsilon$ and for a.e. $\omega$. Therefore,
\begin{equation}\label{C-main_pure}
 m  \int\limits_{\mathbb R^d} (u_3^\varepsilon(x))^2 dx+\frac{1}{\varepsilon^2} \int\limits_{\mathbb R^d} \int\limits_{\mathbb R^d} \,  a (z)   ( u_3^\varepsilon (x-\varepsilon z) -  u_3^\varepsilon(x))^2 dz dx < K
\end{equation}
For the sake of definiteness assume that $B=[-1,1]^d$. The cubes of other size can be considered in exactly the same way.
Let $\phi(s)$ be an even $C_0^\infty(\mathbb R)$ function such that $0\leq \phi\leq 1$, $\phi(s)=1$ for $|s|\leq 1$,
  $\phi(s)=0$ for $|s|\geq 2$, and $|\phi'(s)|\leq 2$.    Denote  $\tilde u_3^\varepsilon(x)= \phi(|x|)u_3^\varepsilon(x)$.
  It is straightforward to check that
  \begin{equation}\label{C-main_modi1}
 m  \int\limits_{\mathbb R^d} (\tilde u_3^\varepsilon(x))^2 dx+\frac{1}{\varepsilon^2} \int\limits_{\mathbb R^d} \int\limits_{\mathbb R^d} \,  a (z)   (\tilde u_3^\varepsilon (x-\varepsilon z) - \tilde u_3^\varepsilon(x))^2 dz dx <  K
\end{equation}
We also choose $\mathcal{R}$ in such a way that $\int_{|z|\leq \mathcal{R}}a(z)dz\geq \frac12$ and introduce
$$
\tilde a(z) ={\bf 1}_{\{|z|\leq \mathcal{R}\}}\,a(z)\,\Big(\int_{|z|\leq \mathcal{R}}a(z)dz\Big)^{-1}.
$$
Then
  \begin{equation}\label{C-main_cut}
 m  \int\limits_{\mathbb R^d} (\tilde u_3^\varepsilon(x))^2 dx+\frac{1}{\varepsilon^2} \int\limits_{\mathbb R^d} \int\limits_{\mathbb R^d} \,  \tilde a (z)   (\tilde u_3^\varepsilon (x-\varepsilon z) - \tilde u_3^\varepsilon(x))^2 dz dx <  K.
\end{equation}
Letting $\tilde B = [-\pi, \pi]^d$, we denote by $\hat u_3^\varepsilon(x)$  the $\tilde B$ periodic extension  of
$\tilde u_3^\varepsilon(x)$.
For the extended function we have
  \begin{equation}\label{C-main_per}
 m  \int\limits_{\tilde B} (\hat u_3^\varepsilon(x))^2 dx+\frac{1}{\varepsilon^2} \int\limits_{\tilde B} \int\limits_{\mathbb R^d} \,  \tilde a (z)   (\hat u_3^\varepsilon (x-\varepsilon z) - \hat u_3^\varepsilon(x))^2 dz dx <  K.
\end{equation}
The functions $e_k(x) = \frac{1}{(2 \pi)^{d/2}} e^{ikx}, \; k \in Z^d$, form an orthonormal basis in $L^2(B)$, and
$$
\hat u_3^\varepsilon(x) = \sum_k \alpha_k^\varepsilon e_k(x), \quad  \hat u_3^\varepsilon (x-\varepsilon z) = \sum_k \alpha_k^\varepsilon e^{-i\varepsilon kz} e_k(x);
$$
$$
\| \hat u_3^\varepsilon(x)\|^2 = \sum_k (\alpha_k^\varepsilon)^2, \quad \|\hat u_3^\varepsilon (x-\varepsilon z) -
\hat u_3^\varepsilon(x) \|^2 =\sum_k (\alpha_k^\varepsilon)^2 |e^{-i\varepsilon k z} - 1|^2.
$$
Then inequality \eqref{C-main} is equivalent to the following bound
\begin{equation}\label{AAA1}
\frac{1}{\varepsilon^2} \sum_k (\alpha_k^\varepsilon)^2 \, \int\limits_{\mathbb R^d} \,  \tilde a (z)  |e^{-i\varepsilon k z} - 1|^2 dz < C.
\end{equation}

\begin{lemma}\label{Propc1c2}
For any $k \in Z^d$ and any $0<\varepsilon<1$ there exist constants $C_1, \ C_2$ (depending on $d$) such that
\begin{equation}\label{A2}
\int\limits_{\mathbb R^d} \, \tilde a (z)  |e^{-i\varepsilon k z} - 1|^2 dz \ge \min \{ C_1 k^2 \varepsilon^2, \ C_2 \}.
\end{equation}
\end{lemma}
\begin{proof}
For small $\varepsilon$, the lower bound by $C_1 k^2 \varepsilon^2$ follows from the expansion of $e^{-i \varepsilon k z}$ in the neighborhood of 0. For large enough $\varepsilon |k|\ge \varkappa_0>1$ we use the following inequality
$$
\int\limits_{\mathbb R^d} \, \tilde a (z)  |e^{-i\varepsilon k z} - 1|^2 dz \ge c_0 \int\limits_{[0,1]^d}  |e^{-i\varepsilon k z} - 1|^2 dz \ge c_0 \big(2-\frac{2}{\varkappa_0}\big)^d.
$$
\end{proof}

Let us consider a sequence $\varepsilon_j \to 0$. Using inequalities \eqref{AAA1}-\eqref{A2} we will construct now for any $\delta>0$ a finite $2 \delta$-net covering all elements of the sequence  $u_3^{\varepsilon_j}$. For any $\delta>0$ we take $|k_0|$ and $j_0$ such that
\begin{equation}\label{A3}
\frac{C}{\delta} < C_1 |k_0|^2 < \frac{C_2}{\varepsilon_{j_0}^2},
\end{equation}
where $C,\, C_1, \, C_2$ are the same constants as in \eqref{AAA1}-\eqref{A2}. Then it follows from \eqref{AAA1}-\eqref{A3} that
$$
\sum_{k:|k| \ge |k_0|} C_1 |k_0|^2 (\alpha_k^{\varepsilon_j})^2 < \sum_{k: |k| \ge |k_0|} \min \Big\{ C_1 |k|^2, \, \frac{C_2}{\varepsilon_j^2} \Big\} \, (\alpha_k^{\varepsilon_j})^2 < C \quad \mbox{ for any } \; j>j_0.
$$
Consequently we obtain the uniform bound on the tails of $\hat u_3^{\varepsilon_j}$ for all $j>j_0$:
\begin{equation}\label{A4}
\sum_{k:|k| \ge |k_0|}  (\alpha_k^{\varepsilon_j})^2 <  \frac{C}{C_1 |k_0|^2} < \delta.
\end{equation}
Denote by ${\cal H}_{k_0} \subset L^2(\tilde B)$ a linear span of basis vectors $\{ e_k, \ |k|<|k_0| \}$. Evidently, it is a finite-dimensional subspace. Then we have
$$
\hat u_3^\varepsilon = w_{k_0}^\varepsilon + \sum_{k:|k| \ge |k_0|}  \alpha_k^{\varepsilon} e_k, \quad \mbox{ where } \; w_{k_0}^\varepsilon = P_{{\cal H}_{k_0}} u_3^\varepsilon.
$$
Since we already know from Lemma \ref{Bound} that the functions $\hat u_3^{\varepsilon_j}$ are uniformly bounded in $L^2(\tilde B)$, then the functions  $w_{k_0}^{\varepsilon_j}$ are also uniformly bounded. Therefore there exists in  ${\cal H}_{k_0}$ a finite $\delta$-net covering the functions $\{ w_{k_0}^{\varepsilon_j}, \, j>j_0 \}$. Estimate \eqref{A4} implies that the same net will be the $2 \delta$-net for the functions  $\{\hat u_3^{\varepsilon_j}, \, j>j_0 \}$. We need to add to this net  $j_0$ elements to cover first $j_0$ functions  $\hat u_3^{\varepsilon_j}, \, j=1, \ldots, j_0$.

Thus we constructed the finite $2 \delta$-net for any $\delta>0$ which proves the compactness of  $\{\hat u_3^{\varepsilon}  \}$ as $\varepsilon \to 0 $  in $L^2(\tilde B)$.

Since $u_3^{\varepsilon}(x)=\hat u_3^{\varepsilon}(x)$ for $x\in B$, we conclude that the family $\{u_3^{\varepsilon}\}$ is compact
in $L^2(B)$.  In the same way one can show that this family is compact on any cube $B=[-L,L]^d$.
This completes the proof of Lemma.
\end{proof}

\begin{lemma}\label{l_u3small}
The following limit relation holds:  $\|u_3^\eps\|_{L^2(\mathbb R^d)}\to 0$, as $\eps\to0$.
\end{lemma}
\begin{proof}
We go back to formula \eqref{u3-main}. On the right-hand side of this equality we have the inner product of two sequences $F^\varepsilon$ and $u_3^\varepsilon$ Since the  sequence $F^\eps \rightharpoonup  0$ weakly in $L^2(B)$, and the sequence $u_3^\varepsilon$ is compact in $L^2(B)$,  the product $(F^\varepsilon, u_3^\varepsilon) \to 0$ as $\varepsilon \to 0$.
Therefore, both integrals on the left-hand side of  \eqref{u3-main} also tend to zero as $\varepsilon \to 0$, and we obtain that $\| u_3^\varepsilon \|_{L^2(\mathbb R^d)} \to 0, \ \varepsilon \to  0$.
\end{proof}

Denote by $\Theta$ the matrix $\Theta = D_1 - D_2$, where  $D_1, \, D_2$ are defined by \eqref{J_1}, \eqref{D_2}. Our next goal is to show that $D_1 - D_2$ is a positive definite matrix.

\begin{proposition}
The matrix $\Theta = D_1 - D_2$ is positive definite:
\begin{equation}\label{Positive}
 \Theta \ = \  \frac12 \, \int\limits_{\mathbb R^d}  \int\limits_{\Omega} \big(z\otimes z - z \otimes \zeta_{-z} (0, \omega ) \big) \,  a (z) \, \mu (0, \omega ) \mu ( -z, \omega)   \, dz \,   d P(\omega) > 0.
\end{equation}
\end{proposition}

\begin{proof}
We recall that $\varkappa^\delta(\omega)$ stands for a unique solution of equation \eqref{A-delta}. Letting
$\varkappa_\eta^\delta(\omega)=\eta\cdot\varkappa^\delta(\omega)$, $\eta\in\mathbb R^d\setminus \{0\}$,
one can easily obtain
\begin{equation}\label{Prop2_eta}
\begin{array}{c}
\displaystyle
\delta \int\limits_\Omega \big(\varkappa_\eta^\delta(\omega)\big)^2\mu(\omega)\, dP(\omega)
 - \int\limits_{\mathbb R^d} \int\limits_\Omega a (z) \mu ( T_z \omega ) \big( \varkappa_\eta^\delta (T_z \omega ) - \varkappa_\eta^\delta ( \omega) \big)  \varkappa_\eta^\delta ( \omega)\mu(\omega) \,  dz \, dP(\omega) \\ \displaystyle
 =   \int\limits_{\mathbb R^d} \int\limits_\Omega (\eta\cdot z) a(z) \varkappa_\eta^\delta(\omega) \mu(T_z \omega) \mu(\omega)  \, dz \, dP(\omega).
\end{array}
\end{equation}
In the same way as in the proof of Proposition \ref{spectrA},
we derive the following relation:
\begin{equation}\label{Prop2_etabis}
\begin{array}{c}
\displaystyle
\delta \int\limits_\Omega \big(\varkappa_\eta^\delta(\omega)\big)^2\mu(\omega)\, dP(\omega)
  +\frac12\int\limits_{\mathbb R^d} \int\limits_\Omega a (z) \mu ( T_z \omega ) \big( \varkappa_\eta^\delta (T_z \omega ) - \varkappa_\eta^\delta ( \omega) \big)^2\mu(\omega) \,  dz \, dP(\omega) \\ \displaystyle
 = - \frac12  \int\limits_{\mathbb R^d} \int\limits_\Omega (\eta\cdot z) a(z)\big( \varkappa_\eta^\delta (T_z \omega ) - \varkappa_\eta^\delta ( \omega) \big) \mu(T_z \omega) \mu(\omega)  \, dz \, dP(\omega).
\end{array}
\end{equation}
According to \eqref{theta} the sequence $\eta\cdot(\varkappa_\eta^{\delta_j} (T_z \omega ) - \varkappa_\eta^{\delta_j} ( \omega))$
converges weakly in $L^2_M $ as $\delta_j\to 0$ to $\eta\cdot\theta(z,\omega)$. Passing to the limit $\delta_j\to0$
in relation \eqref{Prop2_etabis} and considering the lower semicontinuity of the $L^2_M$ norm with respect to the weak
topology, we arrive at the following inequality
\begin{equation}\label{est_ineq}
\frac12\int\limits_{\mathbb R^d} \int\limits_\Omega a (z) \mu ( T_z \omega ) \big(\eta\cdot\theta(z,\omega) \big)^2\mu(\omega) \,  dz \, dP(\omega) \leq
- \frac12  \int\limits_{\mathbb R^d} \int\limits_\Omega (\eta\cdot z) a(z)\big( \eta\cdot\theta(z,\omega) \big) \mu(T_z \omega) \mu(\omega)  \, dz \, dP(\omega).
\end{equation}
Therefore,
$$
\Theta \eta\cdot\eta= \frac12 \, \eta_i\eta_j\int\limits_{\mathbb R^d}  \int\limits_{\Omega} \big(z^i z^j - z^i  \zeta^j_{-z} (0, \omega ) \big) \,  a (z) \, \mu (0, \omega ) \mu ( -z, \omega)   \, dz \,   d P(\omega)
$$
$$
=\frac12  \int\limits_{\mathbb R^d} \int\limits_\Omega \big((\eta\cdot z)^2+(\eta\cdot z) (\eta\cdot \theta(z,\omega))\big) \,  a (z) \, \mu (0, \omega ) \mu ( z, \omega)   \, dz \,   d P(\omega).
$$
Combining the latter relation with \eqref{est_ineq} we obtain
$$
\Theta \eta\cdot\eta\geq \frac12  \int\limits_{\mathbb R^d} \int\limits_\Omega \big((\eta\cdot z)+(\eta\cdot z) (\eta\cdot \theta(z,\omega))\big)^2 \,  a (z) \, \mu (0, \omega ) \mu ( z, \omega)   \, dz \,   d P(\omega).
$$
Since  $\theta(z, \omega)$ is a.s. a function of sublinear growth in $z$, we conclude that   $ \eta\cdot\theta(z, \omega) \not \equiv  \eta\cdot z$, consequently the integral on the right-hand side here is strictly positive.
This yields the desired positive definiteness.
\end{proof}

\section{Estimation of the remainder $ \phi_\varepsilon $}\label{s_estrem}

In this section we consider the remainder $ \phi_\varepsilon (x, \omega)$ given by (\ref{14}) and  prove that $\|\phi_\varepsilon\|_{L^2(\mathbb R^d)}$ vanishes a. s. as $\varepsilon \to 0$.

\begin{lemma}\label{reminder}

Let $u_0 \in {\cal{S}}(\mathbb R^d)$. Then a.s.
\begin{equation}\label{fi}
\| \phi_\varepsilon (\cdot, \omega) \|_{L^2(\mathbb R^d)} \ \to \ 0 \quad \mbox{ as } \; \varepsilon \to 0.
\end{equation}
\end{lemma}

\begin{proof}
The first term in (\ref{14}) can be written as
$$
\phi_\varepsilon^{(1)} (x, \omega) =  \int\limits_{\mathbb R^d} dz \ a (z) \mu \Big( \frac{x}{\varepsilon}, \omega \Big) \mu \Big( \frac{x}{\varepsilon} -z, \omega \Big)  \int_0^{1} \ \Big( \nabla \nabla u_0(x - \varepsilon z t) - \nabla \nabla u_0(x) \Big) z \otimes z (1-t) \ dt.
$$
It doesn't depend on the random corrector $ \theta$ and  can be considered exactly in the same way as in  \cite[Proposition 5 ]{PiZhi17}.
Thus we have
\begin{equation}\label{phi_1bis}
\| \phi_\varepsilon^{(1)} \|_{L^2(\mathbb R^d)} \to 0 \quad \mbox{ as } \; \varepsilon \to 0.
\end{equation}
Let us denote by $\phi_\varepsilon^{(2)}$ the sum of the second and the third terms in (\ref{14}):
\begin{equation}\label{reminder-2}
\begin{array}{rl} \displaystyle
\!\!\!\!&\hbox{ }\!\!\!\!\!\!\!\!\!\!\!\!\phi_\varepsilon^{(2)} (x, \omega) =\\[3mm]
&\!\!\!\!\!\!\!\!\! \displaystyle
\mu \big( \frac{x}{\varepsilon},\omega \big)  \int\limits_{\mathbb R^d} \ a (z)  \mu \big( \frac{x}{\varepsilon} -z, \omega \big)   \theta \big(\frac{x}{\varepsilon}\!-\!z,\omega \big) \Big(  \frac{1}{\varepsilon} \big(\nabla u_0(x- \varepsilon z) - \nabla u_0(x)\big) + z \, \nabla \nabla u_0(x) \Big)\, dz.
\end{array}
\end{equation}
We take sufficiently large  $L>0$ such that supp $\, u_0 \subset \{|x|<\frac12 L \}$ and estimate $\phi_\varepsilon^{(2)} (x, \omega)$ separately in the sets $\{|x|<L\}$ and $\{|x|>L\}$.
If $|x|>L$, then $u_0(x) = 0$.  Since $a(z)$ has a finite second moment in $\mathbb R^d$, for any $c>0$ we have
\begin{equation}\label{ineqz2}
\frac{1}{\varepsilon^2} \int\limits_{|z|> \frac{c}{\varepsilon}}  a (z) \, dz = \frac{1}{\varepsilon^2} \int\limits_{|z|> \frac{c}{\varepsilon}}  a (z) \frac{z^2}{z^2} \, dz \le \frac{1}{c^2} \int\limits_{|z|> \frac{c}{\varepsilon}} a (z) z^2 \, dz \to 0 \quad \mbox{as } \; \varepsilon \to 0.
\end{equation}
Therefore,
\begin{equation}\label{r-2out}
\begin{array}{l}
\displaystyle
\| \phi_\varepsilon^{(2)} \, \chi_{|x|>L} \|^2_{L^2(\mathbb R^d)}
=\!\! \int\limits_{|x|>L}  \Big(\!\! \int\limits_{|x - \varepsilon z|< \frac12 L}\!\! \frac{1}{\varepsilon}
\mu \big( \frac{x}{\varepsilon},\omega \big)  a (z)  \mu \big( \frac{x}{\varepsilon} -z, \omega \big)   \theta \big(\frac{x}{\varepsilon}\!-\!z,\omega \big) \nabla u_0(x- \varepsilon z) \, dz \Big)^2 dx
\\[3mm] \displaystyle
< \alpha_2^4  \, \Big( \frac{1}{\varepsilon^2}
 \int\limits_{|z|> \frac{L}{2\varepsilon}} \ a (z) \, dz \, \Big)^2 \|\varepsilon  \theta \big(\frac{y}{\varepsilon},\omega \big)
 \nabla u_0(y)\|_{L^2(\mathbb R^d)}^2 \to 0;
\end{array}
\end{equation}
Here we have also used the limit relation  $\|  \varepsilon  \theta \big(\frac{y}{\varepsilon},\omega) \nabla u_0(y) \|_{L^2(\mathbb R^d)} \to 0$ that is ensured by Proposition \ref{1corrector}.
Denote $\chi_{<L}(x) = \chi_{\{|x|<L\}}(x)$ and represent the function $\phi_\varepsilon^{(2)} (x,\omega) \, \chi_{<L}(x)$ as  follows:
\begin{equation}\label{r-2in-bis}
\phi_\varepsilon^{(2)} (x, \omega) \, \chi_{<L} (x) = \gamma_\varepsilon^{<} (x, \omega) + \gamma_\varepsilon^{>} (x, \omega),
\end{equation}
where
\begin{equation}\label{r-2in}
\begin{array}{l} \displaystyle
\gamma_\varepsilon^{<} (x, \omega)  =\mu \big( \frac{x}{\varepsilon},\omega \big) \chi_{<L}(x)\\[3mm]
 \displaystyle
\times\int\limits_{|\varepsilon z|< 2L } \ a (z)  \mu \big( \frac{x}{\varepsilon} -z, \omega \big)   \theta \big(\frac{x}{\varepsilon}\!-\!z,\omega \big) \Big(  \frac{1}{\varepsilon} \big(\nabla u_0(x- \varepsilon z) - \nabla u_0(x)\big) + z \, \nabla \nabla u_0(x) \Big)\, dz;
\\[9mm]
\displaystyle
\gamma_\varepsilon^{>} (x, \omega)  = \mu \big( \frac{x}{\varepsilon},\omega \big) \chi_{<L}(x) \\[3mm]
 \displaystyle
\times\int\limits_{|\varepsilon z|> 2L } \ a (z)  \mu \big( \frac{x}{\varepsilon} -z, \omega \big)   \theta \big(\frac{x}{\varepsilon}\!-\!z,\omega \big) \Big(  \frac{1}{\varepsilon} \big(\nabla u_0(x- \varepsilon z) - \nabla u_0(x)\big) + z \, \nabla \nabla u_0(x) \Big)\, dz.
\end{array}
\end{equation}
Since $u_0\in C_0^\infty(\mathbb R^d)$,  the Teylor decomposition applies to $\nabla u_0 (x- \varepsilon z)$, and we get
$$
\frac{1}{\varepsilon} \big(\nabla u_0 (x- \varepsilon z) - \nabla u_0(x)\big) + z \, \nabla \nabla u_0(x) = \frac{\varepsilon}{2} \nabla\nabla\nabla u_0 (\xi)\, z \otimes z
$$
with some $\xi \in \mbox{supp} \, u_0$, here the notation $\nabla\nabla\nabla u_0 (\xi)\, z \otimes z$ is used for the vector function
 $(\nabla\nabla\nabla u_0 (\xi)\, z \otimes z)^i=\partial_{x^j}\partial_{x^k}\partial_{x^i}u_0(\xi)z^jz^k$. Then the right-hand side of the first formula in \eqref{r-2in} admits the  estimate
\begin{equation}\label{r-2in1}
\begin{array}{l} \displaystyle
\mu \big( \frac{x}{\varepsilon},\omega \big) \chi_{<L}(x) \Big|\!\!\!\int\limits_{|\varepsilon z|< 2L } \!\!\!\!\!\! a (z)  \mu \big( \frac{x}{\varepsilon} -z, \omega \big)   \theta \big(\frac{x}{\varepsilon}\!-\!z,\omega \big) \Big(  \frac{1}{\varepsilon} \big(\nabla u_0(x- \varepsilon z) - \nabla u_0(x)\big) + z  \nabla \nabla u_0(x)\! \Big) dz \Big|
\\[3mm]
\displaystyle
 \le \frac{\alpha_2^2}{2} \max | \nabla\nabla\nabla u_0 |  \int\limits_{\mathbb R^d } \, \varepsilon  | \theta \big(\frac{x}{\varepsilon}\!-\!z,\omega \big)| \, \chi_{<3L}(x-\varepsilon z) \, a (z) z^2 \, dz.
\end{array}
\end{equation}
Taking into account the relation
\begin{equation}\label{r-2in1add}
\begin{array}{l} \displaystyle
\int\limits_{\mathbb R^d } \Big( \int\limits_{\mathbb R^d } \, \varepsilon  | \theta \big(\frac{x}{\varepsilon}\!-\!z,\omega \big)| \, \chi_{<3L}(x-\varepsilon z) \, a (z) z^2 \, dz \Big)^2 dx
\\[3mm]
\displaystyle
= \int\limits_{\mathbb R^d } a (z_1) z_1^2 dz_1  \int\limits_{\mathbb R^d } a (z_2) z_2^2 dz_2  \int\limits_{\mathbb R^d }  \varepsilon^2  | \theta \big(\frac{x}{\varepsilon}\!-\!z_1,\omega \big)|  | \theta \big(\frac{x}{\varepsilon}\!-\!z_2,\omega \big)| \chi_{<3L}(x-\varepsilon z_1) \chi_{<3L}(x-\varepsilon z_2) dx
\end{array}
\end{equation}
and applying the Cauchy-Schwartz inequality to the last integral on its right hand side
we conclude  with the help of  Proposition \ref{1corrector} that  $\|\gamma_\varepsilon^{<} (x, \omega) \|_{L^2(\mathbb R^d) } \to 0$ as $\varepsilon \to 0$.


If $|x|<L$ and $|\varepsilon z|>2L$, then $|x-\varepsilon z|>L$, and $u_0 (x-\varepsilon z)=0$. The right-hand side of the second formula in \eqref{r-2in} can be rearranged as follows:
\begin{equation}\label{r-2in2}
\begin{array}{l} \displaystyle
\gamma_\varepsilon^{>} (x, \omega) =
\mu \big( \frac{x}{\varepsilon},\omega \big) \chi_{<L}(x) \int\limits_{|z|> \frac{2L}{\varepsilon} }\!\!\!\!  a (z)  \mu \big( \frac{x}{\varepsilon} -z, \omega \big)   \theta \big(\frac{x}{\varepsilon}\!-\!z,\omega \big) \Big( - \frac{1}{\varepsilon} \nabla u_0(x) + z \, \nabla \nabla u_0(x) \Big)\, dz
\\[3mm] \displaystyle
=\mu \big( \frac{x}{\varepsilon},\omega \big) \chi_{<L}(x) \!\!  \int\limits_{|z|> \frac{2L}{\varepsilon} } \!\!\!\!  a (z)  \mu \big( \frac{x}{\varepsilon} -z, \omega \big) \big(  \theta \big(\frac{x}{\varepsilon}\!-\!z,\omega \big) -  \theta \big(\frac{x}{\varepsilon},\omega \big) \big) \Big(\!\! - \frac{1}{\varepsilon} \nabla u_0(x) + z  \nabla \nabla u_0(x)\!\Big) dz
\\[3mm] \displaystyle
+\mu \big( \frac{x}{\varepsilon},\omega \big) \chi_{<L}(x) \!\! \int\limits_{|z|> \frac{2L}{\varepsilon} } \!\!\!\! a (z)  \mu \big( \frac{x}{\varepsilon} -z, \omega \big)   \theta \big(\frac{x}{\varepsilon},\omega \big) \Big( - \frac{1}{\varepsilon} \nabla u_0(x) + z \, \nabla \nabla u_0(x) \Big)\, dz
\end{array}
\end{equation}
The second term on the right-hand side in \eqref{r-2in2} is estimated in the same way as the function $\phi_\varepsilon^{(2)} \, \chi_{|x|>L}$ in \eqref{r-2out}.
Thus the $L^2(\mathbb R^d)$ norm of this term tends to 0 as $\varepsilon \to 0$.

The first term on the right-hand side of \eqref{r-2in2} admits the following upper bound:
\begin{equation}\label{r-2in2bis}
\begin{array}{l}  \displaystyle
\Big| \mu \big( \frac{x}{\varepsilon},\omega \big) \chi_{<L}(x) \int\limits_{|z|> \frac{2L}{\varepsilon} } \ a (z)  \mu \big( \frac{x}{\varepsilon} -z, \omega \big) \zeta_{-z} \big(T_{\frac{x}{\varepsilon}}\omega \big) \Big( - \frac{1}{\varepsilon} \nabla u_0(x) + z \, \nabla \nabla u_0(x) \Big)\, dz \Big|
\\[3mm] \displaystyle
\leq \alpha_2^2 \int\limits_{|z|> \frac{2L}{\varepsilon} } \ a (z) \Big|  \zeta_{-z} \big(T_{\frac{x}{\varepsilon}}\omega \big)\Big|\ \Big| - \frac{1}{\varepsilon} \nabla u_0(x) + z \, \nabla \nabla u_0(x)\Big|\, dz
\\[3mm] \displaystyle
\leq \alpha_2^2 C(L) \int\limits_{|z|> \frac{2L}{\varepsilon} } |z| a (z) \Big|  \zeta_{-z} \big(T_{\frac{x}{\varepsilon}}\omega \big)\Big| \, dz\ \big(\big|  \nabla u_0(x)\big| + \big| \nabla \nabla u_0(x)\big|\big).
\\[3mm] \displaystyle
\leq \alpha_2^2 C(L) \Big(\int\limits_{|z|> \frac{2L}{\varepsilon} } |z|^2 a (z)dz\Big)^\frac12
 \Big(\int\limits_{\mathbb R^d}  a (z)
  \big|\zeta_{-z} \big(T_{\frac{x}{\varepsilon}}\omega \big)\big|^2 \, dz\Big)^\frac12\ \big(\big|  \nabla u_0(x)\big| + \big| \nabla \nabla u_0(x)\big|\big).
\end{array}
\end{equation}
Since $\zeta_{-z} (\omega)\in L^2_M$, we have
$$
\mathbb E\int\limits_{\mathbb R^d}  a (z)
 | \zeta_{-z} (\omega )|^2 \, dz<\infty.
$$
Taking into account the convergence
$$
\int\limits_{|z|> \frac{2L}{\varepsilon} } |z|^2 a (z)dz\to 0,\quad \hbox{as }\eps\to0,
$$
by the Birkhoff ergodic theorem we obtain that the $L^2(\mathbb R^d)$ norm of the first term on the right-hand side of \eqref{r-2in2}
tends to zero a.s., as $\eps\to0$. Therefore, $\|\gamma_\varepsilon^{>} (x, \omega) \|_{L^2(\mathbb R^d) } \to 0$ as $\varepsilon \to 0$.

From  \eqref{r-2in-bis} it follows that $\| \phi_\varepsilon^{(2)}(x,\omega)  \chi_{<L} (x) \|_{L^2(\mathbb R^d)} \to 0$ as $ \varepsilon \to 0$, and together with \eqref{r-2out} this implies that
\begin{equation}\label{rr}
\| \phi_\varepsilon^{(2)}(x,\omega)  \|_{L^2(\mathbb R^d)} \to 0 \quad \mbox{as } \; \varepsilon \to 0.
\end{equation}
Finally, \eqref{fi} follows from \eqref{phi_1bis} and \eqref{rr}.  Lemma is proved.
\end{proof}


\section{Proof of the main results}\label{s_proofmain}

We  begin this section by proving relation \eqref{convergence1} for $f\in \mathcal{S}_0(\mathbb R^d)$. For such $f$ we have
$u_0\in C_0^\infty(\mathbb R^d)$. It follows from \eqref{v_eps}, Proposition \ref{1corrector} and Lemmas \ref{l_u2small}, \ref{l_u3small}
that
\begin{equation}\label{frstconv}
\|w^\eps-u_0\|_{L^2(\mathbb R^d)}\to 0,\quad\hbox{as }\eps\to 0.
\end{equation}
By the definition of $v^\eps$, $u_2^\eps$ and $u_3^\eps$,
$$
(L^\eps-m)w^\eps=(\hat L-m)u_0-m\eps  \theta \Big(\frac x\eps\Big)\cdot\nabla u_0+\phi_\eps
=f-m\eps  \theta \Big(\frac x\eps\Big)\cdot\nabla u_0+\phi_\eps
$$
$$
=(L^\eps-m)u^\eps-m\eps  \theta \Big(\frac x\eps\Big)\cdot\nabla u_0+\phi_\eps.
$$
Therefore,
$$
(L^\eps-m)(w^\eps-u^\eps)=-m\eps  \theta \Big(\frac x\eps\Big)\cdot\nabla u_0+\phi_\eps.
$$
According to Proposition \ref{1corrector} and Lemma \ref{reminder} the $L^2$ norm of the functions on the right-hand side
of the last formula tends to zero as $\eps\to0$.  Consequently,
$$
\|w^\eps-u^\eps\|_{L^2({\mathbb R^d})}\to 0,\quad\hbox{as }\eps\to 0.
$$
Combining this relation with \eqref{frstconv} yields the desired relation \eqref{convergence1} for $f\in\mathcal{S}_0(\mathbb R^d)$.

To complete the proof of Theorem \ref{T1} we should show that the last convergence  holds for any $f\in L^2(\mathbb R^d)$.

For any $f \in L^2(\mathbb R^d)$ there exists $f_\delta \in \mathcal{S}_0$ such that $\| f - f_\delta\|_{L^2(\mathbb R^d)} <\delta$.
Since the operator $(L^\varepsilon - m)^{-1}$ is bounded uniformly in $\varepsilon$, then
\begin{equation}\label{delta_1}
\| u^{\varepsilon}_\delta - u^\varepsilon \|_{L^2(\mathbb R^d)} \le C_1 \delta,
\qquad
\| u_{0,\delta} - u_0 \|_{L^2(\mathbb R^d)} \le C_1 \delta,
\end{equation}
where
$$
u^{\varepsilon} \ = \ (L^{\varepsilon} - m)^{-1} f, \; \; u_{0} \ = \ (\hat L - m)^{-1} f, \; \;
u^{\varepsilon}_\delta \ = \ (L^{\varepsilon} - m)^{-1} f_\delta, \; \; u_{0,\delta} \ = \ (\hat L - m)^{-1} f_\delta.
$$
Recalling that $f_\delta\in\mathcal{S}_0$, we obtain  $\| u^{\varepsilon}_\delta - u_{0, \delta} \|_{L^2(\mathbb R^d)} \to 0 $. Therefore,  by (\ref{delta_1})
$$
\mathop{ \overline{\rm lim}}\limits_{\varepsilon \to 0} \| u^{\varepsilon} - u_0 \|_{L^2(\mathbb R^d)}  \le 2 C_1 \delta
$$
with an arbitrary  $\delta>0$. This implies the desired convergence in \eqref{t1} for an arbitrary $f\in L^2(\mathbb R^d)$
and completes the proof of the main theorem.

\subsection{Proof of Corollary \ref{cor_main}}

Here we assume that the operator $L^{\eps,{\rm ns}}$ is defined by \eqref{L_eps_ns}.  Multiplying equation \eqref{u_eps_nssss}
by $\rho^\eps(x,\omega)=\rho\big(\frac{x}{\eps},\omega\big)=
\mu\big(\frac{x}{\eps},\omega\big)\big(\lambda\big(\frac{x}{\eps},\omega\big)\big)^{-1}$
we obtain
\begin{equation}\label{eq_modfd}
L^{\eps}u_\eps -m\rho^\eps u_\eps=\rho_\eps f,
\end{equation}
where the symmetrized operator $L^{\eps}$ is given by \eqref{L_eps}. 
Letting $\langle\rho\rangle=\mathbb E\bm{\rho}
=\mathbb E\big(\frac{\bm{\mu}}{\bm{\lambda}}\big)$ we consider an auxiliary equation
\begin{equation}\label{eq_ns_aux}
L^{\eps}g_\eps -m\langle\rho\rangle g_\eps=\langle\rho\rangle f.
\end{equation}
By Theorem \ref{T1} the functions $g_\eps$ converge a.s. in $L^2(\mathbb R^d)$, as $\eps\to0$, to a solution of the equation $\hat Lg -m\langle\rho\rangle g=\langle\rho\rangle f$. Our goal is to show that $\|g_\eps-u_\eps\|_{L^2(\mathbb R^d)}\to0$
as $\eps\to0$. To this end we subtract equation \eqref{eq_modfd} from \eqref{eq_ns_aux}.
After simple rearrangements this yields
\begin{equation}\label{eq_ns_alpha}
L^{\eps}\alpha_\eps -m\rho_\eps \alpha_\eps=\big(\langle\rho\rangle-\rho_\eps\big)g_\eps +\big(\langle\rho\rangle-\rho_\eps\big) f.
\end{equation}
with  $\alpha_\eps(x)=g_\eps(x)-u_\eps(x)$. In a standard way one can derive the following estimate
\begin{equation}\label{C_ns_pure}
 m  \int\limits_{\mathbb R^d} (\alpha_\varepsilon(x))^2 dx+\frac{1}{\varepsilon^2} \int\limits_{\mathbb R^d} \int\limits_{\mathbb R^d} \,  a (z)   ( \alpha_\varepsilon (x-\varepsilon z) -  \alpha_\varepsilon(x))^2 dz dx < C.
\end{equation}
As was shown in the proof of Lemma \ref{Convergence}, this estimate implies compactness of the family $\{\alpha_\eps\}$
in $L^2(B)$ for any cube $B$.  Multiplying \eqref{eq_ns_alpha} by $\alpha_\eps$ and integrating the resulting relation
over $\mathbb R^d$ we obtain
\begin{equation}\label{al_al}
\|\alpha_\eps\|^2_{L^2(\mathbb R^d)}\leq C_1
\big|\big((\langle\rho\rangle-\rho_\eps)g_\eps, \alpha_\eps\big)_{L^2(\mathbb R^d)}\big| +\big|\big((\langle\rho\rangle-\rho_\eps) f,\alpha_\eps\big)_{L^2(\mathbb R^d)}\big|
\end{equation}
By the Birkhoff ergodic theorem $(\langle\rho\rangle-\rho_\eps)$ converges to zero weakly in $L^2_{\rm loc}(\mathbb R^d)$.
Considering the boundedness of $(\langle\rho\rangle-\rho_\eps)$ and the properties of $\alpha_\eps$ and $g_\eps$, we conclude that the both terms on the right-hand side in  \eqref{al_al} tend to zero, as $\eps\to0$. So does
$\|\alpha_\eps\|^2_{L^2(\mathbb R^d)}$.  Therefore, $u_\eps$ converges to the solution of equation
$\hat Lu -m\langle\rho\rangle u=\langle\rho\rangle f$. Dividing this equation by $\langle\rho\rangle$, we rewrite
the limit equation as follows
$$
\Big(\mathbb E\big\{\frac{\bm\mu}{\bm\lambda}\big\}\Big)^{-1}Q_{ij}\frac{\partial^2 u}{\partial x_i\partial x_j}-mu=f
$$
with $\Theta$ defined in \eqref{Positive}.  This completes the proof of Corollary.

\noindent
{\large \bf Acknowlegements}\\[2mm]
The work on this project was completed during the visit of Elena Zhizhina at the Arctic University of Norway, campus Narvik. She expresses her gratitude to the colleagues at this university for hospitality.

\end{document}